\documentclass{article}

\usepackage[top=1in,bottom=1in,left=1.25in,right=1.25in]{geometry}
\usepackage{graphicx}
\usepackage{caption}
\usepackage{subcaption}
\usepackage{float} 
\usepackage{lmodern}
\usepackage[T1]{fontenc}
\usepackage[english,activeacute]{babel}
\usepackage{mathtools}
\usepackage{amsmath,amsthm,amssymb,amsfonts}
\usepackage{mathrsfs}
\usepackage{stmaryrd} 
\usepackage[noadjust]{cite} 
\usepackage{enumerate}
\usepackage{verbatim} 
\usepackage[usenames,dvipsnames,svgnames,table]{xcolor}
\usepackage[utf8]{inputenc} 
\usepackage{soul} 
\usepackage{tikz-cd} 
\usetikzlibrary{babel} 
\usepackage{setspace} 
\usepackage{array, multicol, multirow, makecell}
\usepackage{subfiles}
\usepackage[nottoc]{tocbibind} 
\usepackage{longtable}
\usepackage{tikz}
\usetikzlibrary{positioning}
\usepackage{dynkin-diagrams}
\usepackage{authblk}
\usepackage{abstract}
\usepackage{bbm}

\usepackage[pdftex,backref=page]{hyperref} 

\hypersetup{linktocpage=true,colorlinks=true,citecolor=blue,urlcolor=blue}
\newcommand{\doi}[1]{\href{https://doi.org/#1}{\ttfamily #1}}

\usepackage{import}
\usepackage{xifthen}
\usepackage{pdfpages}
\usepackage{transparent}

\let\OLDthebibliography\thebibliography
\renewcommand\thebibliography[1]{
	\OLDthebibliography{#1}
	\setlength{\parskip}{0pt}
}

\newtheorem{theorem}{Theorem}[section]
\newtheorem{corollary}[theorem]{Corollary}
\newtheorem{proposition}[theorem]{Proposition}
\newtheorem{lemma}[theorem]{Lemma}
\newtheorem{conjecture}[theorem]{Conjecture}
\theoremstyle{definition}
\newtheorem{definition}[theorem]{Definition}
\newtheorem{example}[theorem]{Example}
\newtheorem{remark}[theorem]{Remark}

\DeclareMathOperator{\Hom}{Hom}
\DeclareMathOperator{\End}{End}
\DeclareMathOperator{\Aut}{Aut}

\DeclareMathOperator{\Ad}{Ad}
\DeclareMathOperator{\ad}{ad}
\DeclareMathOperator{\Der}{Der}

\DeclareMathOperator{\Mat}{Mat}

\newcommand{\N}{\mathbb{N}}

\newcommand{\R}{\mathbb{R}}
\newcommand{\C}{\mathbb{C}}

\newcommand{\frg}{\mathfrak{g}}
\newcommand{\frh}{\mathfrak{h}}

\newcommand{\frk}{\mathfrak{k}}

\newcommand{\frz}{\mathfrak{z}}

\newcommand{\GL}{\mathrm{GL}}

\newcommand{\SO}{\mathrm{SO}}

\newcommand{\Sp}{\mathrm{Sp}}

\renewcommand{\d}{\mathrm{d}}

\DeclarePairedDelimiter{\escal}{\langle}{\rangle}

\allowdisplaybreaks
\usepackage{paralist}
\newcommand{\Sl}{\mathfrak{sl}}
\newcommand{\Span}[1]{\operatorname{Span}\{#1\}}
\newcommand{\lie}[1]{\mathfrak{#1}}
\newcommand{\su}{\lie{su}}
\DeclareMathOperator{\Tr}{tr}

\newcommand{\g}{\lie g}
\usepackage{orcidlink}

\title{\textbf{
Pseudo-K\"ahler and hypersymplectic structures on semidirect products}}
\author{Diego Conti \orcidlink{0000-0001-6812-4411} and Alejandro Gil-García \orcidlink{0000-0002-9370-241X}}
\affil{\normalsize Dipartimento di Matematica\\
	Università di Pisa\\
	Largo Bruno Pontecorvo, 5, 56127 Pisa PI, Italy\\
	diego.conti@unipi.it}
\affil{\normalsize Beijing Institute of Mathematical Sciences and Applications (BIMSA)\\
        No.\ 544, Hefangkou Village, Huaibei Town, Huairou District, Beijing 101408, China\\
        alejandrogilgarcia@bimsa.cn}
\date{\today}

\begin{document}

\maketitle

\begin{abstract}

We study left-invariant pseudo-K\"ahler and hypersymplectic structures on semidirect products $G\rtimes H$; we work at the level of the Lie algebra $\mathfrak{g}\rtimes\mathfrak{h}$. In particular we consider the structures induced on $\mathfrak{g}\rtimes\mathfrak{h}$ by existing pseudo-K\"ahler structures on $\mathfrak{g}$ and $\mathfrak{h}$; we classify all semidirect products of this type with $\mathfrak{g}$ of dimension $4$ and $\mathfrak{h}=\mathbb{R}^2$. In the hypersymplectic setting, we consider a more general construction on semidirect products. We construct a large class of hypersymplectic Lie algebras whose underlying complex structure is not abelian as well as non-flat hypersymplectic metrics on $k$-step nilpotent Lie algebras for every $k\geq3$.\bigskip

\medskip


\end{abstract}

\clearpage

\tableofcontents

\section*{Introduction}

Left-invariant metrics on a Lie group provide a setting for the study of geometric structures which is particularly suited to the construction of explicit examples: since all computations can be performed at the Lie algebra level, the PDE's characterizing the integrability conditions reduce to linear equations. Among left-invariant metrics, a widely studied class consists of semidirect products. Indeed, the study of Riemannian Einstein homogeneous spaces of negative scalar curvature, thanks to \cite{BL23,Lau10,Heber:noncompact}, reduces to the study of standard solvmanifolds of Iwasawa type, namely semidirect products $\g\rtimes_D\R$, where $\g$ is nilpotent and $D$ symmetric. This does not extend to indefinite signature, but even in this context semidirect products still provide a large class of Einstein metrics \cite{ContiRossi:IndefiniteNilsolitons}. Semidirect products have also been studied in the context of non-negative Ricci curvature (see \cite{MiatelloMetrics}) and ad-invariant metrics, where one exploits the fact that any Lie algebra $\g$ yields a canonical semidirect product $\g^*\rtimes\g$, with $\g$ acting on $\g^*$ via the coadjoint representation, admitting an ad-invariant metric of neutral signature (see \cite{Ova16}). Finally, it should be noted that every Lie algebra which is neither solvable nor semisimple has a non-trivial Levi decomposition, i.e.\ it is the semidirect product of its radical and a semisimple subalgebra; this has been exploited for instance in \cite{Fino_Raffero_2019} to construct closed $\mathrm{G}_2$-structures on non-solvable Lie groups.\medskip

The structures studied in this paper belong to the class of pseudo-K\"ahler metrics. Whilst positive-definite K\"ahler left-invariant metrics, or more generally homogeneous, are fairly well understood (see \cite{Lic88,DN88}), pseudo-K\"ahler invariant metrics show much greater flexibility. This is already evident in real dimension $4$, where Ovando's classification lists eleven distinct Lie algebras carrying a pseudo-K\"ahler metric, but only six admitting a definite K\"ahler metric (see \cite{Ova06}). Moreover, K\"ahler nilpotent Lie algebras are necessarily abelian by \cite{Hasegawa}, but this does not hold in the pseudo-K\"ahler case, as one can see from Ovando's classification. Pseudo-K\"ahler structures on Lie algebras have been also studied in \cite{CFU04,BS12,LU24}.\medskip

Pseudo-K\"ahler Lie groups find application in the construction of homogeneous quaternion-K\"ahler manifolds via the c-map (see \cite{Man21}). They can also   be used in the construction of Sasaki-Einstein solvmanifolds of indefinite signature, yielding some of the few known examples of left-invariant metrics admitting a Killing spinor (see \cite{Conti_Rossi_SegnanDalmasso_2023}).\medskip

Complex and pseudo-K\"ahler structures on semidirect products have already been considered in \cite{CCO15}, where in particular the authors characterize the situation in which the product almost complex structure on a semidirect product of two complex Lie algebras is integrable. In addition, \cite{CCO15} contains several examples of semidirect products $\g\rtimes\frh$ in dimension $6$ endowed with a complex structure for which $\frh$ is totally real. A variation of the notion of semidirect product was also considered in \cite{Valencia} in the context of special K\"ahler Lie algebras.\medskip

A particular class of pseudo-K\"ahler metrics is formed by hypersymplectic structures, introduced by Hitchin in \cite{Hit90}. A hypersymplectic structure on a $4n$-dimensional manifold is determined by a complex structure and a product structure that anticommute, together with a compatible metric such that the associated $2$-forms are closed. The holonomy of the metric is contained in the non-compact Lie group $\Sp(2n,\R)$, which is the split-real form of $\Sp(2n,\C)$. Hence, hypersymplectic manifolds are neutral-signature analogues of hyperk\"ahler manifolds, whose holonomy is contained in the compact real form $\Sp(n)$. Due to the common complexification of the holonomy groups, many facts from hyperk\"ahler geometry carry over to hypersymplectic manifolds. In particular, hypersymplectic manifolds are complex symplectic and Ricci-flat. Moreover, since $\Sp(2n,\R)\subset\mathrm{U}(n,n)$, hypersymplectic are particular examples of neutral Calabi-Yau manifolds. These have been studied in \cite{FPPS04,LU21}.\medskip

Left-invariant hypersymplectic structures have been widely studied. In \cite{And06}, the author characterizes hypersymplectic Lie algebras in terms of two Lie algebras equipped with flat torsion-free connections and parallel symplectic forms. This allows him to classify $4$-dimensional hypersymplectic Lie algebras. A procedure to construct hypersymplectic structures on $\R^{4n}$ beginning with affine-symplectic data on $\R^{2n}$ was given in \cite{AD06}. These hypersymplectic structures are shown to be invariant by a $3$-step nilpotent double Lie group and the resulting metrics are complete and not necessarily flat. The first $4$-step nilpotent examples of hypersymplectic Lie algebras were obtained in \cite{BGL21}, where the authors provide a method to construct hypersymplectic structures from the data of a pseudo-K\"ahler and a complex symplectic structure.\medskip

Two outstanding questions at the time of writing are whether it is possible to construct $k$-step nilpotent hypersymplectic Lie algebras for every $k\geq5$ and if there exists a $2$-step nilpotent hypersymplectic Lie algebra with non-flat metric. We will answer the first question in this paper, and leave the second open.\medskip

The key tool of this paper is a method to construct pseudo-K\"ahler structures on some semidirect products of Lie algebras (see Theorem~\ref{thm:construction}). We start with two pseudo-K\"ahler Lie algebras $\frg$ and $\frh$ and define a natural almost pseudo-Hermitian structure of the semidirect product $\frg\rtimes\frh$. Then we determine the conditions on the representation defining the semidirect product to ensure that the almost complex structure on $\frg\rtimes\frh$ is parallel; we call the the resulting object a \emph{pseudo-K\"ahler extension} (see Definition~\ref{def:pseudokahlerextension}). This allows us to construct several examples of pseudo-K\"ahler Lie algebras in dimension $6$ and $8$, starting both with an abelian $\frh$ and a non-abelian $\frh$ (see Section~\ref{sec:examples}).\medskip

We then consider a special case of $6$-dimensional pseudo-K\"ahler extensions of the form $\frg\rtimes\R^2$, and provide a classification (not exactly up to isometry, but using a stronger notion introduced in Definition~\ref{def:equivalence}). It turns out that of the eleven $4$-dimensional Lie algebras which admit a pseudo-K\"ahler structure classified in \cite{Ova06}, only three admit a pseudo-K\"ahler extension. Taking into account the metrics, this gives rise to four families of pseudo-K\"ahler Lie algebras of dimension $6$ (in one of which the parameter can be eliminated by rescaling).\medskip

We also study hypersymplectic structures on semidirect products of pseudo-K\"ahler Lie algebras as considered above. Examples of this type have been constructed in \cite{CPO11} with $\frg=\frh=\R^{2n}$ for any $n\in\N$; they all satisfy a special condition which we call \emph{Kodaira type}, namely they are $2$-step nilpotent with $J$-invariant center of dimension equal to half the dimension of the Lie algebra. We obtain $2$-step nilpotent hypersymplectic Lie algebras in every dimension $4n$ which are neither of Kodaira type nor equipped with an abelian complex structure. However, all the $2$-step nilpotent examples that we obtain are flat. Together with the existing examples in the literature, this leads us to conjecture that every $2$-step nilpotent hypersymplectic Lie algebra is flat.\medskip

Moreover, we are able to produce $k$-step nilpotent hypersymplectic Lie algebras, both flat and non-flat, for any $k\geq3$. To our knowledge, these are the first examples of $k$-step nilpotent hypersymplectic Lie algebras for $k\geq5$. Examples with $k=3$ and $k=4$ have been constructed in \cite{AD06} and \cite{BGL21}, respectively.\medskip

We note that the metrics constructed in this paper are of potential interest in the study of Einstein metrics. Indeed, all the hypersymplectic examples are Ricci-flat (indeed, neutral Calabi-Yau), and so are most of the pseudo-K\"ahler metrics, though not all. In the same spirit as \cite{Conti_Rossi_SegnanDalmasso_2023}, one can then apply a construction of \cite{BerardBergery} and obtain a pseudo-K\"ahler-Einstein bundle in two dimensions higher, on which one can consider the Sasaki-Einstein cone.

\subsection*{Acknowledgements}

The authors thank the referees for their valuable comments. D.\ C.\ would like to acknowledge the project PRIN 2022MWPMAB ``Interactions between Geometric Structures and Function Theories'', GNSAGA of INdAM and  the MIUR Excellence Department Project awarded to the Department of Mathematics, University of Pisa, CUP I57G22000700001. A.\ G.\ is supported by the German Science Foundation (DFG) under Germany's Excellence Strategy  --  EXC 2121 ``Quantum Universe'' -- 390833306.

\section{Construction of pseudo-K\"ahler structures}

In this section we introduce some  fundamental objects that will appear throughout the paper; in particular, pseudo-K\"ahler structures, semidirect products, and $\mathrm{U}(p,q)$-structures on a Lie algebra. We then introduce conditions for a semidirect product of two pseudo-K\"ahler Lie algebras to admit an induced pseudo-K\"ahler structure.\medskip

Given a manifold $M$ of dimension $n$ and a group $K\subset\GL(n,\R)$, a $K$-structure is a principal subbundle of the bundle of frames with fibre $K$ (see \cite{Kobayashi95}); notice that we have replaced with $K$ the more conventional symbol $G$, in order to reserve the latter for the ambient manifold. We will be interested in the particular case where $K=\mathrm{U}(p,q)$ is the group that preserves a complex structure on $\R^n=\R^{2(p+q)}$ and a scalar product $g$ of signature $(2p,2q)$ such that $g(J\cdot,J\cdot)=g$. A $\mathrm{U}(p,q)$-structure may be identified with a pair $(g,J)$, where $g$ is a pseudo-Riemannian metric on $M$ of signature $(2p,2q)$, and $J$ an almost complex structure such that $g(J\cdot,J\cdot)=g$. We point out that the terminology ``almost Hermitian'' is also used in the literature, but mostly reserved to the case $q=0$. A $\mathrm{U}(p,q)$-structure is pseudo-K\"ahler if $J$ is a complex structure and $\omega=g(J\cdot,\cdot)$ is closed, or equivalently if $J$ is parallel relative to the Levi-Civita connection of $g$.\medskip

Analogous definitions can be given on a Lie algebra $\g$. Thus, a pseudo-K\"ahler Lie algebra is a triple $(\g,g,J)$, where $\g$ is a real Lie algebra, $g$ a non-degenerate scalar product on $\g$, $J\colon\g\to\g$ an almost complex structure such that $g(J\cdot,J\cdot)=g$, and one of the following equivalent conditions hold:
\begin{compactenum}
\item $J$ is parallel relative to the Levi-Civita connection $\nabla$ of $g$; or
\item $J$ is a complex structure, i.e.\ the Nijenhuis tensor $N_J$ is zero, and the $2$-form $\omega=g(J\cdot,\cdot)$ is closed.
\end{compactenum}

It is clear that a pseudo-K\"ahler structure on a Lie algebra defines a left-invariant pseudo-K\"ahler structure on a Lie group with that Lie algebra. However, we shall perform all computations at the Lie algebra level, with no need to consider the group.\medskip

We aim at constructing pseudo-K\"ahler structures on a semidirect product. This means that we have pseudo-K\"ahler Lie algebras $(\g,g,J_g)$ and $(\lie h,h,J_h)$, and in addition a homomorphism $\varphi\colon\frh\to\Der(\g)$. We can then define the semidirect product
\[\tilde\g:=\g\rtimes_\varphi\frh, \quad [X+A,Y+B]_{\tilde\g}:=[X,Y]_\g+\varphi(A)Y-\varphi(B)X+[A,B]_\frh.\]
Here and in the sequel, we adopt the convention that $X,Y,Z$ denote elements of $\g$, and $A,B,C$ denote elements of $\frh$.\medskip

As a vector space, $\tilde \g$ is isomorphic to $\g\oplus\frh$, so it has an induced $\mathrm{U}(p,q)$-structure \begin{equation*}
    \tilde g(X+A,Y+B):=g(X,Y)+h(A,B), \quad \tilde J(X+A):=J_g(X)+J_h(A).
\end{equation*}

Notice that the subscripts in $J_g$ and $J_h$ have been chosen to suggest the fact that they are complex structures on the Lie algebra denoted by the corresponding (gothic) letter.

\begin{remark}\label{remark:2-step}
    If $\frg$ and $\frh$ are abelian, then $\tilde\frg$ is at most $2$-step solvable. If in addition $\varphi(A)\varphi(B)=0$ for all $A,B\in\frh$, then $\tilde\frg$ is $2$-step nilpotent. Indeed, the derived Lie algebra $[\tilde\frg,\tilde\frg]$ is spanned by the vectors $[X+A,Y+B]_{\tilde\g}=\varphi(A)Y-\varphi(B)X\in\frg$. Since $\frg$ is abelian, then $[\tilde\frg,\tilde\frg]$ is also abelian, which means that $\tilde\frg$ is $2$-step solvable. Using that $\varphi(A)\varphi(B)=0$ for all $A,B\in\frh$, we get \begin{align*}
        [X+A,[Y+B,Z+C]_{\tilde\g}]_{\tilde\g}&=[X+A,\varphi(B)Z-\varphi(C)Y]_{\tilde\g}\\
        &=\varphi(A)(\varphi(B)Z-\varphi(C)Y)=0,
    \end{align*} which means that $\tilde\frg$ is $2$-step nilpotent.
\end{remark}

We aim at constructing pseudo-K\"ahler structures on semidirect products. To that end, we need to introduce more notation and a lemma.\medskip

Having fixed $\g$ and $g$, for any $f\colon\g\to\g$, we will denote by $f^*$ the adjoint operator, i.e.\ $g(f\cdot,\cdot)=g(\cdot,f^*\cdot)$, and write $f=f^s+f^a$, where $f^s$ and $f^a$ denote the $g$-symmetric and $g$-antisymmetric parts of $f$, i.e.\ $f^s=\frac12 (f+f^*)$, $f^a=\frac12 (f-f^*)$. With this notation, any $\mathrm{U}(p,q)$-structure satisfies $J^*=-J$.

\begin{lemma}\label{lemma:tilde_nabla_XY}
    Let $X,Y,Z\in\frg$ and $A,B,C\in\frh$. Then
    \begin{gather*}
    \tilde{g}(\tilde{\nabla}_XY,Z+C)=g(\nabla^g_XY,Z)+g(\varphi(C)^sX,Y);\\
        \tilde{\nabla}_XB=-\varphi(B)^sX;\quad
    \tilde\nabla_A B=\nabla^h_A B; \quad     \tilde{\nabla}_AY=\varphi(A)^aY,
    \end{gather*}
where $\tilde\nabla$, $\nabla^g$, $\nabla^h$ denote the Levi-Civita connections respectively on $(\tilde \g,\tilde g)$, $(\g, g)$ and $(\lie h,h)$.
\end{lemma}

\begin{proof}
Koszul's formula for the Levi-Civita connection gives
\begin{align*}
2\tilde{g}(\tilde{\nabla}_XY,Z+C)&=\tilde g([X,Y]_{\tilde\g},Z+C)-\tilde g(Y,[X,Z+C]_{\tilde\g})-\tilde g(X,[Y,Z+C]_{\tilde\g})\\
&= g([X,Y]_\g,Z)- g(Y,[X,Z]_\g)- g(X,[Y,Z]_\g)\\
&\quad+\tilde g([X,Y]_\g,C)-\tilde g(Y,[X,C]_{\tilde\g})-\tilde g(X,[Y,C]_{\tilde\g})\\
&= 2g(\nabla^g_X Y,Z)+g(\varphi(C)X,Y)+g(\varphi(C)Y,X),
\end{align*}
giving the first equation. The second follows similarly from
\begin{align*}
2\tilde{g}(\tilde{\nabla}_XB,Z+C)&= g([X,B]_{\tilde\g},Z)-\tilde g(B,[X,C]_{\tilde\g})-\tilde g(X,[B,Z]_{\tilde\g})\\
&=-g(\varphi(B)X,Z)-g(X,\varphi(B)Z).
\end{align*}
For the last two, we compute
\begin{gather*}
2\tilde{g}(\tilde{\nabla}_AB,Z+C)
=h([A,B]_{\frh},C)-h(B,[A,C]_{\frh})-h(A,[B,C]_{\frh})=2h(\nabla_A B,C),\\
2\tilde{g}(\tilde{\nabla}_AY,Z+C)
=g([A,Y]_{\tilde \frg},Z)-g(Y,[A,Z]_{\tilde \frg})-\tilde g(A,[Y,Z+C]_{\tilde\frg})=2g(\varphi(A)^a,Z).\qedhere
\end{gather*}
\end{proof}

We can now prove:
\begin{theorem}\label{thm:construction}
    Let $(\frg,g,J_g)$ and $(\frh,h,J_h)$ be pseudo-K\"ahler Lie algebras and let $\varphi:\frh\to\Der(\frg)$ be a representation. Then $(\tilde\frg,\tilde g,\tilde J)$ is a pseudo-K\"ahler Lie algebra if and only if \begin{itemize}
        \itemsep 0em
        \item $J_g\circ\varphi(A)^s=\varphi(J_hA)^s$,
        \item $J_g\circ\varphi(A)^a=\varphi(A)^a\circ J_g$,
    \end{itemize} for all $A\in\frh$.
\end{theorem}

\begin{proof}
The pseudo-K\"ahler condition is equivalent to $\tilde\nabla \tilde J=0$. Using Lemma~\ref{lemma:tilde_nabla_XY} and  $\tilde g(\tilde J\cdot,\cdot)=-\tilde g(\cdot,\tilde J\cdot)$, we find \begin{align*}
        \tilde g((\tilde\nabla_X\tilde J)Y,Z+C)&=\tilde g(\tilde\nabla_X\tilde J Y,Z+C)-\tilde g(\tilde J\tilde\nabla_XY,Z+C)\\
        &=\tilde g(\tilde\nabla_XJ_gY,Z+C)+\tilde g(\tilde\nabla_XY,J_gZ+J_hC)\\
        &=g(\nabla^g_XJ_gY,Z)+g(\varphi(C)^sX,J_gY)\\
        &\quad+g(\nabla^g_XY,J_gZ)+g(\varphi(J_hC)^sX,Y).
    \end{align*}

    Notice that since $\nabla^gJ_g=0$,
    $$g(\nabla^g_XJ_gY,Z)+g(\nabla^g_XY,J_gZ)=g(\nabla^g_XJ_gY,Z)-g(J_g\nabla^g_XY,Z)=0.$$

    Hence, we obtain $$\tilde g((\tilde\nabla_X\tilde J)Y,Z+C)=g((\varphi(J_hC)^s-J_g\varphi(C)^s)X,Y).$$
The other component of $\tilde\nabla_X\tilde J$ is determined by
\[(\tilde\nabla_X\tilde J)B=\tilde\nabla_X\tilde JB-\tilde J(\tilde\nabla_XB)=\tilde\nabla_XJ_hB-\tilde J(-\varphi(B)^sX)=-\varphi(J_hB)^sX+J_g\varphi(B)^sX.\]
On the other hand, $\tilde\nabla_A\tilde J$ is determined by
     \begin{align*}
        (\tilde\nabla_A\tilde J)Y&=\tilde\nabla_A\tilde JY-\tilde J(\tilde\nabla_AY)=\tilde\nabla_AJ_gY-\tilde J(\varphi(A)^aY)\\
        &=\varphi(A)^aJ_gY-J_g\varphi(A)^aY,\\
        (\tilde\nabla_A\tilde J)B&=\tilde\nabla_A\tilde J B-\tilde J\tilde\nabla_AB=\nabla^h_AJ_hB-J_h\nabla^h_AB=0,
    \end{align*}
where we have used $\nabla^hJ_h=0$.
\end{proof}

\begin{definition}
\label{def:pseudokahlerextension}
A pseudo-K\"ahler Lie algebra $(\tilde\frg,\tilde g,\tilde J)$ constructed as in Theorem~\ref{thm:construction} will be called a \emph{pseudo-K\"ahler extension} of $(\frg,g,J_g)$ by $(\frh,h,J_h)$.
\end{definition}

We will occasionally refer to pseudo-K\"ahler extensions simply as extensions in order to avoid repetitions.

\begin{example}
\label{example:withthenotation}
Let us consider an example in real dimension $4$. Take two copies of the abelian $2$-dimensional Lie algebra, with a positive-definite and a negative-definite pseudo-K\"ahler structure, i.e.\
\begin{align*}
\frg&=\Span{e_1,e_2},& g&=e^1\otimes e^1+e^2\otimes e^2,&J_ge_1&=e_2,\\
\frh&=\Span{a_1,a_2},& h&=-a^1\otimes a^1-a^2\otimes a^2,&J_ha_1&=a_2.
\end{align*}

The notation $\{e^1,e^2\}$ represents the basis of $\g^*$ dual to $\{e_1,e_2\}$; similarly for $a^1,a^2$. This convention will be used again in Section~\ref{sec:classification}.\medskip

Any representation $\varphi:\frh\to\Der(\frg)$ satisfying the conditions of Theorem~\ref{thm:construction} will have $\varphi(a_1)^s=\varphi(a_2)^s=0$. The space of endomorphisms commuting with $J_g$ and with symmetric part equal to zero is spanned by $J_g$ itself, so up to a unitary change of basis we can assume $\ad(a_1)=\lambda J_g$ and $\ad(a_2)=0$. The resulting $4$-dimensional Lie algebra takes the form
\[\tilde\g=\Span{e_1,e_2,e_3,e_4}, \quad e_3=a_1, e_4=a_2,\]
with
\[\d e^1=-\lambda e^2\wedge e^3,\,\d e^2=\lambda e^1\wedge e^{3},\,\d e^3=0=\d e^4.\]
Throughout the paper, we will write more succinctly
\[\tilde \g=(-\lambda e^{23},\lambda e^{13},0,0),\]
with notation adapted from \cite{Salamon:ComplexStructures}. The pseudo-K\"ahler structure is given by
\[\tilde Je_1=e_2, \tilde Je_3=e_4,\quad \tilde g=e^1\otimes e^1+e^2\otimes e^2-e^3\otimes e^3-e^4\otimes e^4.\]
The Lie algebra $\tilde\g$ is abelian if $\lambda=0$, and otherwise isomorphic to the Lie algebra denoted by $\mathfrak{rr}'_{3,0}$ in \cite{Ova06} with a flat pseudo-K\"ahler structure.
\end{example}

\begin{remark}
The pseudo-K\"ahler Lie algebras constructed in Theorem~\ref{thm:construction} have a non-trivial $\tilde J$-invariant ideal of the same real dimension as $\frg$. This shows in particular that not every pseudo-K\"ahler Lie algebra can be obtained from this construction.
\end{remark}

\begin{remark}
    The above construction also works in the case of \emph{para-K\"ahler} Lie algebras. These are triples $(\frg,g,E)$ where $E$ is an integrable para-complex structure, $g(E\cdot,E\cdot)=-g$ and $\omega:=g(E\cdot,\cdot)$ is closed.
\end{remark}

Recall that a complex structure $J$ is called \emph{abelian} if $[JX,JY]=[X,Y]$ for all $X,Y\in\frg$. Then we have the following result.

\begin{lemma}\label{lemma:abelian_cx}
    In the situation of Theorem~\ref{thm:construction}, the complex structure $\tilde J$ is abelian if and only if both $J_g$ and $J_h$ are abelian and $\varphi(J_hA)^aJ_g=\varphi(A)^a$ for all $A\in\frh$.
\end{lemma}

\begin{proof}
    Let $X,Y\in\frg$ and $A,B\in\frh$. Then \begin{itemize}
        \itemsep 0em
        \item $[\tilde J X,\tilde J Y]_{\tilde\g}=[X,Y]_{\g}$ if and only if $J_g$ is abelian,
        \item $[\tilde J A,\tilde J B]_{\tilde\g}=[A,B]_{\frh}$ if and only if $J_h$ is abelian,
        \item $[\tilde J A,\tilde J Y]_{\tilde\g}=[A,Y]_{\tilde\g}$ if and only if $\varphi(J_hA)J_g=\varphi(A)$ for all $A\in\frh$.
    \end{itemize}

    Decomposing the last equation into its $g$-symmetric and $g$-antisymmetric part, and using that
    \[\varphi(A)^s=-J_g\varphi(J_hA)^s
    =(-J_g\varphi(J_hA)^s)^*
    =\varphi(J_hA)^sJ_g,\] we get $\varphi(J_hA)^aJ_g=\varphi(A)^a$.
\end{proof}

In general, the representation $\varphi$ may have a non-trivial kernel. It turns out that if the kernel is $J_h$-invariant, it can be factored out before performing the semidirect product construction, giving rise to a semidirect product $\g\rtimes_{\varphi'}(\frh/\ker\varphi)$ where the induced map $\varphi'$ is injective. More generally, we have the following:

\begin{proposition}
\label{prop:factoroutkernel}
In the hypotheses of Theorem~\ref{thm:construction}, let $\lie k$ be a non-degenerate $J_h$-invariant ideal in $\ker\varphi$. Then both $\lie k$ and $\lie h/\lie k$ are pseudo-K\"ahler, the induced  map $\varphi'\colon \lie h/\lie k\to\Der(\lie g)$ still satisfies the hypotheses of Theorem~\ref{thm:construction}, and we have an exact sequence of pseudo-K\"ahler Lie algebras
\[0\to \lie k \to \g\rtimes_{\varphi}\lie h\to \g\rtimes_{\varphi'} \frac{\frh}{\lie k}\to 0.\]
\end{proposition}

\begin{proof}
Denote by $\nabla^k$ the Levi-Civita connection on $\lie k$. Then for $A,B,C\in\lie k$ we have
\[h( \nabla^k_A B,C) =h( \nabla^h_A B,C).\]
therefore, $\nabla^kJ_h=0$.\medskip

Similarly, the projection map $\pi\colon\lie h\to\lie h/\lie k$ restricts to an isomorphism of vector spaces $\lie k^\perp\cong \lie h/\lie k$; denoting by $p$ the metric induced on $\lie h/\lie k$, we have
\[p(\nabla^{\pi}_{\pi(A)} \pi(B),\pi(C)) =h(\nabla^h_A B,C), \quad A,B,C\in \lie k^\perp,\]
where $\nabla^\pi$ indicates the covariant derivative on $\lie h/\lie k$. Therefore, $\nabla^\pi J_{\frh/\frk}=0$.\medskip

Now the fact that $\varphi'$ satisfies the hypotheses of Theorem~\ref{thm:construction} is straightforward.
\end{proof}

\begin{remark}
At the Lie group level, a semidirect product $\lie g\rtimes_\varphi\lie h$ determines a semidirect product $G\rtimes_\Phi H$, where $G$ and $H$ are the simply connected Lie groups with Lie algebras $\g$, $\lie h$, and $\Phi\colon G\to\Aut(H)$ is obtained by lifting $\varphi$ to a homomorphism of Lie groups $\tilde\varphi\colon G\to\Aut(\lie h)$ and then composing with the natural inclusion $\Aut(\lie h)\to\Aut(H)$. We may then interpret the long exact sequence as a principal bundle $$\begin{tikzcd}
K \arrow[r] & G\rtimes_\Phi H \arrow[d] \\
            & G\rtimes_{\Phi'}(H/K)
\end{tikzcd}$$
\end{remark}

A special class of semidirect products appears in the study of Einstein Riemannian solvmanifolds under the name of \emph{standard} solvmanifolds (see \cite{Heber:noncompact}). This condition was generalized to the pseudo-Riemannian case in \cite{ContiRossi:IndefiniteNilsolitons}: a standard decomposition of a Lie algebra $\tilde \g$ endowed with a metric $\tilde g$ is an orthogonal decomposition $\tilde\g=\frg\rtimes\frh$, where $\frg$ is a nilpotent ideal and $\frh$ an abelian subalgebra.\medskip

We can use Theorem~\ref{thm:construction} to  construct standard pseudo-K\"ahler extensions of a fixed nilpotent Lie algebra (though nilpotency is not essential in what follows); motivated by Proposition~\ref{prop:factoroutkernel}, we illustrate this in the case where $\varphi$ is injective.
\begin{corollary}\label{cor:abelianextension}
If $(\lie g,g,J)$ is a pseudo-K\"ahler Lie algebra and $\lie h\subset\Der(\lie g)$ an abelian subalgebra such that:
\begin{compactenum}
\item $f^*\in\lie h$ for all $f\in\lie h$, so that
\[\lie h=\lie h_0\oplus \lie h_1,\]
where $\lie h_0$ consists of $g$-antisymmetric derivations and $\lie h_1$ consists of $g$-symmetric derivations;
\item $\lie h_0$ has even dimension;
\item for every $A\in\lie h_1$, $J\circ A\in\lie h_1$;
\item $J\circ A=A\circ J$ for all $A$ in $\lie h_0$.
\end{compactenum}
Then $\lie g\rtimes_\iota\lie h$ has a pseudo-K\"ahler structure, where $\iota\colon\lie h\to \Der(\g)$ is the inclusion.
\end{corollary}

\begin{proof}
Define an almost complex structure $J_h$ on $\lie h_1$ by $J_h(A)=J\circ A$, and extend it to $\lie h$ by choosing an arbitrary complex structure on $\lie h_0$. Since $\frh$ is abelian, any compatible metric defines a pseudo-K\"ahler structure.\medskip

Relative to the inclusion $\iota\colon \lie h\to\Der(\lie g)$, the conditions of Theorem~\ref{thm:construction} hold.
\end{proof}

\begin{remark}
In the situation of Corollary~\ref{cor:abelianextension}, if we further assume that $\Tr X$ and $\Tr XY$ vanish for all $X,Y\in\lie h$, we see that the metric on $\lie g\rtimes\lie h$ is Ricci-flat by \cite[Proposition 4.1]{ContiRossi:IndefiniteNilsolitons}.
\end{remark}

A natural question is when two pseudo-K\"ahler extensions of a Lie algebra $\g$ obtained by the method of Theorem~\ref{thm:construction} should be regarded as different. We will make use of the following criterion:

\begin{proposition}[\cite{AzencottWilson2,Conti_Rossi_SegnanDalmasso_2023}]
\label{prop:pseudoAzencottWilson}
Let $K$ be a subgroup of $\SO(r,s)$ with Lie algebra $\lie k$ and $\tilde{\g}$ a Lie algebra of the form $\tilde{\g}=\g\rtimes \lie{h}$ endowed with a $K$-structure. Let $\chi\colon\lie h\to\Der(\g)$ be a Lie algebra homomorphism such that, extending $\chi(A)$ to $\tilde{\g}$ by declaring it to be zero on $\lie h$,
\begin{equation}
\label{eqn:adXstar}
\chi(A)-\tilde{\ad}(A)\in\lie k, \quad [\chi (A),\tilde{\ad}(B)]=0,\ A,B\in\lie h.
\end{equation}

Let $\tilde{\g}^*$ be the Lie algebra $\g\rtimes_\chi\lie{h}$. If $\tilde{G}$ and $\tilde{G}^*$ denote the connected, simply connected Lie groups with Lie algebras $\tilde{\g}$ and $\tilde{\g}^*$, with the corresponding left-invariant $K$-structures, there is an isometry from $\tilde{G}$ to $\tilde{G}^*$, whose differential at $e$ is the identity of $\g\oplus\lie{h}$ as a vector space, mapping the $K$-structure on $\tilde{G}$ into the $K$-structure on $\tilde{G}^*$.
\end{proposition}

\begin{remark}
Azencott and Wilson proved in \cite[Lemma 4.2]{AzencottWilson2} that in the context of Riemannian solvmanifolds with negative sectional curvature and with notation as in Proposition~\ref{prop:pseudoAzencottWilson}, the group $\tilde G^*$ is closed in the group of isometries of $\tilde G$; this led them to define the notion of \emph{modification} for solvmanifolds with negative (sectional) curvature. Given a Lie group $S$ with Lie algebra $\lie s$ acting simply transitively by isometries on a Riemannian solvmanifold with negative curvature, $\lie s'$ is a  modification of $\lie s$ if it is a subalgebra of the normalizer of $\lie s$ in the Lie algebra of the group of isometries $I$ and the connected Lie subgroup with Lie algebra $\lie s'$ is closed in  $I$ and acts simply transitively. In the context of this paper, closedness in $I$ is less relevant; the use of Proposition~\ref{prop:pseudoAzencottWilson} is to establish an algebraic criterion which allows us to reduce redundancy in classifying pseudo-K\"ahler extensions. For this reason, we will not use the notion of modification in this paper, but a more ad-hoc definition.
\end{remark}

Beside identifying pseudo-K\"ahler extensions related by symmetries as in Proposition~\ref{prop:pseudoAzencottWilson}, we also identify extensions related by isomorphisms of $\g$ and $\frh$. More precisely:

\begin{definition}
\label{def:equivalence}
Given pseudo-K\"ahler Lie algebras $(\g,g,J_g)$, $(\g',g',J'_g)$, $(\frh,h,J_h)$ and $(\frh',h',J_h')$, and given $\varphi\colon\frh\to\Der(\g)$, $\varphi'\colon\frh'\to\Der(\g')$ satisfying the conditions of Theorem~\ref{thm:construction}, we will say that the pseudo-K\"ahler extensions $\g\rtimes_\varphi\frh$ and $\g'\rtimes_{\varphi'}\frh'$ are \emph{AW-related} (with reference to Azencott and Wilson) if there are Lie algebra isomorphisms
\[f_g\colon\g\to\g',\quad f_h\colon\frh\to\frh',\]
respecting the pseudo-K\"ahler structures, i.e.\ $f_g^*g'=g$, $f_g\circ J_g=J_g'\circ f_g$, $f_h^*h'=h$, $f_h\circ J_h=J_h'\circ f_h$; and such that for all $A\in\frh$, $A'\in\frh'$
\[\big(f_g\varphi(A)f_g^{-1}-\varphi'(f_h(A))\big)^s=0, \quad [f_g\varphi(A) f_g^{-1},\varphi'(A')]=0.\]
We will say a pseudo-K\"ahler extension is \emph{trivial} if it is AW-related to one with $\varphi=0$, i.e.\ a direct product.
\end{definition}

The definition of AW-related extensions implies that $f_g\varphi(A)f_g^{-1}-\varphi'(f_h(A))$ is $g'$-antisymmetric for all $A\in\frh$, and therefore, by the hypotheses on $\varphi,\varphi'$, an element of the Lie algebra $\lie k\cong\lie u(p,q)$ of $g'$-antisymmetric endomorphisms of $\g'$ that commute with $J'_g$. It follows then from Proposition~\ref{prop:pseudoAzencottWilson} that, at the Lie group level, two AW-related extensions $(\tilde G,\tilde g,\tilde J)$, $(\tilde G',\tilde g',\tilde J')$ are related by an isometry of pseudo-Riemannian manifolds that respects the complex structures.

\begin{remark}
    A pseudo-K\"ahler extension of $(\frg,g,J_g)$ by $(\frh,h,J_h)$ is trivial if and only if $\varphi(A)$ is $g$-antisymmetric for all $A\in\frh$.
\end{remark}

\begin{remark}
AW-relatedness as defined in Definition~\ref{def:equivalence} is not an equivalence relation. For instance, consider the pseudo-K\"ahler structure on $\R^4=\escal{e_1,e_2,e_3,e_4}$ given by $$J_ge_1=e_2,J_ge_3=e_4\quad\text{and}\quad g=e^1\otimes e^1+e^2\otimes e^2-e^3\otimes e^3-e^4\otimes e^4,$$ and the one given in Example~\ref{example:withthenotation} on
$\R^2=\escal{a_1,a_2}$. We can define two pseudo-K\"ahler extensions $\R^4\rtimes_{\varphi_1}\R^2$, $\R^4\rtimes_{\varphi_2}\R^2$ with
\[\varphi_1(a_1)=\left(
\begin{array}{cccc}
 0 & -1 & 0 & 1 \\
 1 & 0 & -1 & 0 \\
 0 & -1 & 0 & 1 \\
 1 & 0 & -1 & 0 \\
\end{array}
\right),  \quad \varphi_2(a_1)=\left(
\begin{array}{cccc}
 0 & 1 & 0 & 0 \\
 -1 & 0 & 0 & 0 \\
 0 & 0 & 0 & -1 \\
 0 & 0 & 1 & 0 \\
\end{array}
\right), \quad \varphi_1(a_2)=0=\varphi_2(a_2).\]

These extensions are not AW-related: if they were, there would be an automorphism $f_g\colon\R^4\to\R^4$ such that $f_g\varphi_2(a_1)f_g^{-1}$ commutes with $\varphi_1(a_2)$. Since both $\varphi_1(a_1)$ and $f_g\varphi_2(a_1)f_g^{-1}$ lie in $\su(1,1)$, and the centralizer of $\varphi_1(a_1)$ in $\su(1,1)$ is $1$-dimensional, the two elements can only commute if they are proportional, which is absurd because one is invertible and the other one is nilpotent. Thus, the two extensions are not AW-related; however, each of them is trivial, i.e.\ AW-related to a direct product.
\end{remark}

\begin{remark}
Given pseudo-K\"ahler extensions $\g\rtimes_\varphi\lie h$, $\g\rtimes_{\varphi'}\lie h$, the corresponding simply connected Lie groups may be isometric even if the extensions are not AW-related. For instance we will see in Example~\ref{ex:isometric_not_AW-related} that one can construct a non-trivial pseudo-K\"ahler extension $\tilde \g=\g\rtimes_\varphi\lie h$ which is flat and $2$-step nilpotent. The direct product is also flat and $2$-step nilpotent. By \cite{Guediri1994}, the $2$-step nilpotent condition implies that the two corresponding Lie groups are geodesically complete, and  being flat, they are globally isometric. However, the extensions are not AW-related because $\tilde \g=\g\rtimes_\varphi\lie h$ is non-trivial. Moreover, if one considers the equivalence relation generated by AW-relatedness, every pseudo-K\"ahler extension in the same equivalence class as the direct product $\g\times\frh$ is trivial; hence, $\tilde\g$ and $\g\times\frh$ are in different equivalence classes.
\end{remark}

\section{Examples of dimension 6 and 8}\label{sec:examples}

In this section we describe several explicit examples of pseudo-K\"ahler Lie algebras constructed using Theorem~\ref{thm:construction}. We are interested in particular in dimensions $6$ and $8$, although some of these examples can be also generalized to arbitrary dimensions. We will distinguish the cases in which the Lie algebra $\frh$ is abelian and those where it is not; we will see that with this method one can produce non-flat examples starting with flat Lie algebras.

\subsection{Case of abelian \texorpdfstring{$\mathfrak{h}$}{h}}

In the following examples, we take $\frh$ to be abelian; in the first three, $\frg$ is also abelian.

\begin{example}
    Let $\frh=\R^2=\escal{a_1,a_2}$ be the abelian K\"ahler Lie algebra with the Euclidean metric and complex structure $J_ha_1=a_2$, and let $\frg=\R^4=\escal{e_1,e_2,e_3,e_4}$ be a pseudo-K\"ahler abelian Lie algebra with complex structure and metric given by $$J_ge_1=e_2,J_ge_3=e_4\quad\text{and}\quad g=e^1\otimes e^1+e^2\otimes e^2-e^3\otimes e^3-e^4\otimes e^4.$$

    We define the representation $\varphi:\frh\to\Der(\frg)=\Mat_4(\R)$ by $$\varphi(a_1)=\left(\begin{array}{rrrr}
        1 & 1 & 1 & 1 \\
        1 & -1 & 1 & -1 \\
        -1 & -1 & -1 & -1 \\
        -1 & 1 & -1 & 1
    \end{array}\right)\quad\text{and}\quad\varphi(a_2)=\left(\begin{array}{rrrr}
        -1 & 1 & -1 & 1 \\
        1 & 1 & 1 & 1 \\
        1 & -1 & 1 & -1 \\
        -1 & -1 & -1 & -1
    \end{array}\right).$$

    The map $\varphi$ defined in this way satisfies the conditions in Theorem~\ref{thm:construction} and $\varphi(a_1)$ and $\varphi(a_2)$ commute, hence $(\tilde\frg,\tilde g,\tilde J)$ is a $6$-dimensional pseudo-K\"ahler Lie algebra. Furthermore, the matrices $\varphi(a_1)$ and $\varphi(a_2)$ satisfy $\varphi(a_1)^2=\varphi(a_2)^2=\varphi(a_1)\varphi(a_2)=0$. Then, by Remark~\ref{remark:2-step}, $\tilde\frg$ is $2$-step nilpotent. Moreover, a computation shows that the metric $\tilde g$ is non-flat.
\end{example}

\begin{remark}
    Notice that every pseudo-K\"ahler metric on a nilpotent Lie algebra is Ricci-flat by \cite[Lemma~6.4]{FinoPartonSalamon}.
\end{remark}

\begin{example}
    Let $\frh=\frg=\R^4$ both with complex structure and metric given by $$J_ge_1=e_3,J_ge_2=e_4\quad\text{and}\quad g=e^1\otimes e^1-e^2\otimes e^2+e^3\otimes e^3-e^4\otimes e^4.$$

    We define the representation $\varphi:\frh\to\Der(\frg)$ by $$\varphi(a_1)=\varphi(a_2)=\left(\begin{array}{rrrr}
        1 & 2 & -1 & 0 \\
        0 & 1 & 0 & 1 \\
        1 & 0 & -1 & 0 \\
        0 & -1 & 2 & -1
    \end{array}\right),\quad\varphi(a_3)=\varphi(a_4)=\left(\begin{array}{rrrr}
        0 & 1 & 2 & 1 \\
        1 & 0 & -1 & 0 \\
        0 & 1 & 0 & 1 \\
        -1 & 2 & 1 & 0
    \end{array}\right).$$

    The map $\varphi$ defined like this satisfies the conditions of Theorem~\ref{thm:construction}, so we have a pseudo-K\"ahler structure on $\tilde\frg$. Moreover, since the derived algebra of $\tilde\frg$ is given by $[\tilde\frg,\tilde\frg]=\escal{e_1,e_2,e_3,e_4}$, $\tilde\frg$ is $2$-step solvable. The metric $\tilde g$ in this example is also non-flat.
\end{example}

\begin{example}
    Let $\frh=\R^2$ be the abelian K\"ahler Lie algebra with metric $h=-\mathbbm{1}_2$ and complex structure $J_ha_1=a_2$. Let $\frg=\R^6$ be the abelian pseudo-K\"ahler Lie algebra with metric $$g=e^1\otimes e^1+e^2\otimes e^2+e^3\otimes e^3+e^4\otimes e^4-e^5\otimes e^5-e^6\otimes e^6$$ and complex structure $J_ge_{2j-1}=e_{2j}$ for $j=1,2,3$. Consider the representation $\varphi:\frh\to\Der(\frg)$ given by $$\varphi(a_1)=\left(\begin{array}{rrrrrr}
        1 & 0 & 0 & 0 & 0 & 1 \\
        0 & -1 & 0 & 0 & 1 & 0 \\
        0 & 0 & 0 & 0 & 0 & 0 \\
        0 & 0 & 0 & 0 & 0 & 0 \\
        0 & -1 & 0 & 0 & 1 & 0 \\
        -1 & 0 & 0 & 0 & 0 & -1
    \end{array}\right),\quad\varphi(a_2)=\left(\begin{array}{rrrrrr}
        0 & 1 & 0 & 0 & -1 & 0 \\
        1 & 0 & 0 & 0 & 0 & 1 \\
        0 & 0 & 0 & 0 & 0 & 0 \\
        0 & 0 & 0 & 0 & 0 & 0 \\
        1 & 0 & 0 & 0 & 0 & 1 \\
        0 & -1 & 0 & 0 & 1 & 0
    \end{array}\right).$$

    The map $\varphi$ satisfies the conditions of Theorem~\ref{thm:construction}. Moreover, $\varphi(a_1)^2=\varphi(a_2)^2=\varphi(a_1)\varphi(a_2)=0$. Then, by Remark~\ref{remark:2-step}, $\tilde\frg$ is $2$-step nilpotent. This metric is also non-flat.
\end{example}

Next we consider some cases where $\frg$ is non-abelian, and hence the space of derivations is more restricted.

\begin{example}\label{ex:8-dim-3-step}
    Let $\frh=\R^2=\escal{a_1,a_2}$ be the abelian K\"ahler Lie algebra with the Euclidean metric and complex structure $J_ha_1=a_2$. Let $(\frg,g,J_g)$ be the $6$-dimensional pseudo-K\"ahler Lie algebra taken from \cite[Section~3.1]{Smo13}, where $\frg$ is a $3$-step nilpotent Lie algebra with non-zero brackets $$[e_1,e_2]=e_4,\quad[e_2,e_3]=e_6,\quad[e_2,e_4]=e_5.$$

    The complex structure and the metric are given by $$J_ge_1=e_2, J_ge_3=e_4, J_ge_5=e_6\quad\text{and}\quad g=-e^1\odot e^5-e^2\odot e^6-e^3\otimes e^3-e^4\otimes e^4,$$ where $e^i\odot e^j:=e^i\otimes e^j+e^j\otimes e^i.$\medskip

    We define the representation $\varphi:\frh\to\Der(\frg)$ by $$\varphi(a_1)e_1=\varphi(a_2)e_2=xe_5+ye_6,\quad\varphi(a_1)e_2=-\varphi(a_2)e_1=ye_5-xe_6,$$ where $x,y\in\R$, and $\varphi(a_1)e_j=\varphi(a_2)e_j=0$ for $j=3,4,5,6$. The map $\varphi$ defined in this way satisfies the conditions in Theorem~\ref{thm:construction} and $\varphi(a_1)$ and $\varphi(a_2)$ commute, hence $(\tilde\frg,\tilde g,\tilde J)$ is an $8$-dimensional pseudo-K\"ahler Lie algebra. Furthermore, the matrices $\varphi(a_1)$ and $\varphi(a_2)$ satisfy $\varphi(a_1)^2=\varphi(a_2)^2=\varphi(a_1)\varphi(a_2)=0$. Since $[\tilde\frg,\tilde\frg]=\escal{e_4,e_5,e_6}=[\frg,\frg]$, $\tilde\frg$ is $3$-step nilpotent. This metric is non-flat if $x\neq0$ or $y\neq0$.
\end{example}

\begin{example}\label{ex:8-dim-2-step}
    Let $\frh=\R^2=\escal{a_1,a_2}$ be the abelian K\"ahler Lie algebra with the Euclidean metric and complex structure $J_ha_1=a_2$. Let $(\frg,g,J_g)$ be the $6$-dimensional pseudo-K\"ahler Lie algebra taken from \cite[Section~6.2]{Smo13}, where $\frg$ is a $2$-step nilpotent Lie algebra with non-zero brackets $$[e_1,e_3]=-e_5,\quad[e_2,e_3]=e_6.$$

    The complex structure and the metric are given by $$J_ge_1=-e_2,J_ge_3=e_4,J_ge_5=e_6\quad\text{and}\quad g=e^1\odot e^5-e^2\odot e^6+e^3\otimes e^3+e^4\otimes e^4.$$

    We define the representation $\varphi:\frh\to\Der(\frg)$ by $$\varphi(a_1)e_1=-\varphi(a_2)e_2=xe_5-ye_6,\quad\varphi(a_1)e_2=\varphi(a_2)e_1=ye_5+xe_6,$$ where $x,y\in\R$, and $\varphi(a_1)e_j=\varphi(a_2)e_j=0$ for $j=3,4,5,6$. The map $\varphi$ defined in this way satisfies the conditions in Theorem~\ref{thm:construction} and $\varphi(a_1)$ and $\varphi(a_2)$ commute, hence $(\tilde\frg,\tilde g,\tilde J)$ is an $8$-dimensional pseudo-K\"ahler Lie algebra. Furthermore, the matrices $\varphi(a_1)$ and $\varphi(a_2)$ satisfy $\varphi(a_1)^2=\varphi(a_2)^2=\varphi(a_1)\varphi(a_2)=0$. Since $[\tilde\frg,\tilde\frg]=\escal{e_5,e_6}=[\frg,\frg]$, $\tilde\frg$ is $2$-step nilpotent. The metric $\tilde g$ is non-flat if and only if $x^2+y^2\neq\frac{1}{2}$.\medskip

    The $6$-dimensional pseudo-K\"ahler metric $g$ on $\frg$ defined above is non-flat, so it is interesting to notice that we can choose some $x$ and $y$ on the representation $\varphi$ such that, starting with a non-flat metric, the metric $\tilde g$ on $\tilde\frg$ is flat.
\end{example}

Note that Example~\ref{ex:8-dim-3-step} and Example~\ref{ex:8-dim-2-step} satisfy $\varphi(\frh)\frg\subseteq[\frg,\frg]$ and $[\tilde\frg,\tilde\frg]=[\frg,\frg]$. In the following example this is not the case.

\begin{example}
    We consider the same setting as in Example~\ref{ex:8-dim-2-step}, but we define the representation $\varphi$ as follows $$\varphi(a_1)=\left(\begin{array}{rrrrrr}
        0 & 0 & 0 & 0 & 0 & 0 \\
        0 & 0 & 0 & 0 & 0 & 0 \\
        0 & 0 & 0 & 0 & 0 & 0 \\
        x_{1} & x_{2} & 0 & 0 & 0 & 0 \\
        x_{3} & x_{4} & x_{2} & 0 & 0 & 0 \\
        x_{4} & x_{3} & x_{1} & 0 & 0 & 0
    \end{array}\right),\quad\varphi(a_2)=\left(\begin{array}{rrrrrr}
        0 & 0 & 0 & 0 & 0 & 0 \\
        0 & 0 & 0 & 0 & 0 & 0 \\
        0 & 0 & 0 & 0 & 0 & 0 \\
        x_{2} & -x_{1} & 0 & 0 & 0 & 0 \\
        0 & -2 \, x_{3} & -x_{1} & 0 & 0 & 0 \\
        0 & 0 & x_{2} & 0 & 0 & 0
    \end{array}\right),$$ with $x_1,x_2,x_3,x_4\in\R$. Note that in this case $[\frg,\frg]\subsetneq[\tilde\frg,\tilde\frg]=\escal{e_4,e_5,e_6}$. Nevertheless, $\tilde\frg$ is still $2$-step nilpotent.
\end{example}

\subsection{Case of non-abelian \texorpdfstring{$\mathfrak{h}$}{h}}\label{subsection:non-abelian_h}

First of all, note that $\tilde\nabla_AB=\nabla^h_AB$ for all $A,B\in\frh$ by Lemma~\ref{lemma:tilde_nabla_XY}. This implies that $\tilde R(A,B)C=R^h(A,B)C$, where $R^h$ is the curvature of the metric $h$. Hence, if the metric $h$ is non-flat, so is $\tilde g$.

\begin{example}
    Consider $\frh=\mathfrak{r}'_2$ the $4$-dimensional $2$-step solvable Lie algebra with non-zero Lie brackets $$[a_1,a_3]=a_3,\quad[a_1,a_4]=a_4,\quad[a_2,a_3]=a_4,\quad[a_2,a_4]=-a_3.$$

    This is the real Lie algebra underlying on the complex Lie algebra $\mathfrak{aff}(\C)$. It admits a non-flat pseudo-K\"ahler structure with complex structure $J_ha_1=-a_2,J_ha_3=a_4$ (see \cite[Theorem~4.6]{Ova06}). Consider $\frg=\mathfrak{rh}_3$, with the pseudo-K\"ahler structure
    $$J_ge_1=e_2,J_ge_3=e_4\quad\text{and}\quad g=e^1\odot e^3+e^2\odot e^4-e^1\otimes e^1-e^2\otimes e^2,$$
also from \cite{Ova06}. Define a representation $\varphi:\frh\to\Der(\frg)$ by \begin{align*}
        \varphi(a_1)&=\left(\begin{array}{rrrr}
            0 & 0 & 0 & 0 \\
            0 & 0 & 0 & 0 \\
            -\frac{1}{2} \, y_{1} - \frac{1}{2} \, y_{2} & x_{1} & 0 & 0 \\
            x_{2} & \frac{1}{2} \, y_{1} + \frac{1}{2} \, y_{2} & 0 & 0
        \end{array}\right),\\
        \varphi(a_2)&=\left(\begin{array}{rrrr}
            0 & 0 & 0 & 0 \\
            0 & 0 & 0 & 0 \\
            \frac{1}{2} \, x_{1} + \frac{1}{2} \, x_{2} & y_{1} & 0 & 0 \\
            y_{2} & -\frac{1}{2} \, x_{1} - \frac{1}{2} \, x_{2} & 0 & 0
        \end{array}\right)
    \end{align*} and $\varphi(a_3)=\varphi(a_4)=0$, with $x_1,x_2,y_1,y_2\in\R$. This map satisfies the conditions of Theorem~\ref{thm:construction}. The derived algebra is given by $[\tilde\frg,\tilde\frg]=\escal{e_3,e_4,a_3,a_4}$, so $\tilde\frg$ is $2$-step solvable. Moreover, since $h$ is non-flat, so is $\tilde g$.
\end{example}

\begin{example}
    Let $\frh=\mathfrak{rr}_{3,0}$ be the $4$-dimensional Lie algebra with non-zero Lie bracket $[a_1,a_2]=a_2$ and complex structure $J_ha_1=a_2$, $J_ha_3=a_4$ (see \cite{Ova06}). Let $\frg=\R^2$ be the abelian K\"ahler Lie algebra with the Euclidean metric and complex structure $J_ge_1=e_2$. We define the representation $\varphi:\frh\to\Der(\frg)=\Mat_2(\R)$ by $$\varphi(a_1)=\left(\begin{array}{rr}
        0 & \frac{1}{2} \\
        \frac{1}{2} & 0
    \end{array}\right),\quad\varphi(a_2)=\left(\begin{array}{rr}
        -\frac{1}{2} & \frac{1}{2} \\
        -\frac{1}{2} & \frac{1}{2}
    \end{array}\right),\quad\varphi(a_3)=\varphi(a_4)=0.$$

    The Lie algebra $\frh$ is $2$-step solvable and non-unimodular. The derived series of the Lie algebra $\tilde\frg$ is given by $$\tilde\frg^{(1)}=[\tilde\frg,\tilde\frg]=\escal{e_1,e_2,a_2},\quad\tilde\frg^{(2)}=\escal{e_1+e_2},\quad\tilde\frg^{(3)}=0.$$

    Hence the Lie algebra $\tilde\frg$ is $3$-step solvable and non-unimodular. Since $h$ is non-flat, so is $\tilde g$. Furthermore, one can check that $\tilde g$ is not Ricci-flat.
\end{example}

\begin{example}
    Let $\frh=\mathfrak{rr}_{3,0}$ be as above and let $\frg=\R^4$ be the abelian K\"ahler Lie algebra with the Euclidean metric and complex structure $J_ge_1=e_2,J_ge_3=e_4$. We define the representation $\varphi:\frh\to\Der(\frg)=\Mat_4(\R)$ by $$\varphi(a_1)=\left(\begin{array}{rrrr}
        0 & \frac{1}{2} & 0 & 0 \\
        \frac{1}{2} & 0 & 0 & 0 \\
        0 & 0 & 0 & \frac{1}{2} \\
        0 & 0 & \frac{1}{2} & 0
    \end{array}\right),\quad \varphi(a_2)=\left(\begin{array}{rrrr}
        -\frac{1}{2} & \frac{1}{2} & 0 & 0 \\
        -\frac{1}{2} & \frac{1}{2} & 0 & 0 \\
        0 & 0 & -\frac{1}{2} & \frac{1}{2} \\
        0 & 0 & -\frac{1}{2} & \frac{1}{2}
    \end{array}\right),$$ $$\varphi(a_3)=\left(\begin{array}{rrrr}
        0 & 0 & x & 0 \\
        0 & 0 & 0 & x \\
        -x & 0 & 0 & 0 \\
        0 & -x & 0 & 0
    \end{array}\right),\quad \varphi(a_4)=\left(\begin{array}{rrrr}
        0 & 0 & y & 0 \\
        0 & 0 & 0 & y \\
        -y & 0 & 0 & 0 \\
        0 & -y & 0 & 0
    \end{array}\right),$$ with $x,y\in\R$. The Lie algebra $\frh$ is $2$-step solvable and non-unimodular. Hence we obtain that $\tilde\frg$ is $3$-step solvable and non-unimodular. Moreover, it has $1$-dimensional center given by $\frz(\tilde\frg)=\escal{ya_3-xa_4}$, thus the center is not $\tilde J$-invariant. As in the above example, the metric $\tilde g$ is non-flat.
\end{example}

\begin{example}\label{ex:isometric_not_AW-related}
    Let us consider $\frg=\mathfrak{rh}_3$ with the pseudo-Kähler structure $(g,J_g)$ given by $$g=e^1\odot e^3+e^2\odot e^4\quad\text{and}\quad Je_1=e_2,Je_3=e_4.$$ We also take $\frh=\mathfrak{rh}_3$ with $h=g$ and $J_h=J_g$. Define the representation $\varphi:\frh\to\Der(\frg)$ by $$\varphi(a_1)=\left(\begin{array}{rrrr}
        0 & 0 & 0 & 0 \\
        0 & 0 & 0 & 0 \\
        1 & 0 & 0 & 0 \\
        0 & -1 & 0 & 0
    \end{array}\right),\quad\varphi(a_2)=\left(\begin{array}{rrrr}
        0 & 0 & 0 & 0 \\
        0 & 0 & 0 & 0 \\
        0 & 1 & 0 & 0 \\
        1 & 0 & 0 & 0
    \end{array}\right),\quad\varphi(a_3)=\varphi(a_4)=0.$$

    We have that $[\varphi(a_1),\varphi(a_2)]=0$, $\varphi(a_1)^a=\varphi(a_2)^a=0$ and the Lie algebra $\tilde\g$ is 2-step nilpotent. One can check that the metric $g$ is flat and the metric $\tilde g$ is also flat.
\end{example}

\section{Classification on 6-dimensional semidirect products \texorpdfstring{$\mathfrak{g}\rtimes\mathbb{R}^2$}{g+R2}}\label{sec:classification}

In this section we consider the special case where $\frg$ has dimension $4$ and $\lie h$ is $2$-dimensional and abelian. We exploit the classification of pseudo-K\"ahler Lie algebras of dimension $4$ of \cite{Ova06} and classify all the cases where the hypotheses of Theorem~\ref{thm:construction} are satisfied up to AW-relatedness, as defined in Definition~\ref{def:equivalence}.\medskip

Beside AW-relatedness, we have an obvious symmetry, which is an overall change of sign of the metric on the extension. In light of this symmetry, it makes sense to fix the signature on the $2$-dimensional factor to be positive definite. Thus, throughout the section, we will fix $\frh=\R^2$, with a basis $\{a_1,a_2\}$ such that \begin{equation}\label{eqn:kahleronR2}
    J_ha_1=a_2, \quad h=a^1\otimes a^1+a^2\otimes a^2.
\end{equation}

Since $J_h$ is fixed, we will drop the subscript $g$ in $J_g$ and denote the complex structure on $\g$ by $J$.

\begin{lemma}
\label{lemma:AB1B2}
Let $(\frg,g,J)$ be a pseudo-K\"ahler Lie algebra. Then its pseudo-K\"ahler extensions by an abelian $2$-dimensional Lie algebra $(\frh, h,J_h)$ up to AW-relatedness are in one-to-one-correspondence with triples $(A,B_1,B_2)$ of linear maps $\g\to\g$ such that
\[A=A^*, AJ=-JA, B_i=-B_i^*, B_iJ=JB_i, \]
\[A+B_1, JA+B_2\in\Der(\g)\]
and
\begin{equation}
\label{eqn:symmetricpartofabeliancondition}
[A,B_2]=J[A,B_1],\\
\end{equation}
\begin{equation}
\label{eqn:skewsymmetricpartofabeliancondition}
[B_1,B_2]=2JA^2.
\end{equation}
\end{lemma}

\begin{proof}
Suppose $\varphi$ defines a pseudo-K\"ahler extension; by the conditions of Theorem~\ref{thm:construction}, we can write
\[\varphi(a_1)=A+B_1, \quad \varphi(a_2)=JA+B_2,\]
where $A$ is $g$-symmetric and $B_1$, $B_2$ are $g$-antisymmetric and commute with $J$, and the notation $JA$ represents composition, i.e.\ $(JA)(X)$ stands for $J(A(X))$. In addition, we have that $JA$ is $g$-symmetric, and therefore $JA=A^*J^*=-AJ$.\medskip

Since $\frh$ is abelian, imposing that $\varphi$ is a homomorphism implies that $\varphi(a_1)$ and $\varphi(a_2)$ commute, i.e.\
\[0=[A+B_1,JA+B_2]=[A,JA]+[A,B_2]-[JA,B_1]+[B_1,B_2];\]
taking the $g$-antisymmetric part, we find
\[0=[A,B_2]-[JA,B_1]=[A,B_2]-JAB_1+B_1JA=[A,B_2]-JAB_1+JB_1A
=[A,B_2]-J[A,B_1].\]
Taking the $g$-symmetric part and recalling that $JA=-AJ$, we find
\[0=[A,JA]+[B_1,B_2]=AJA-JAA+[B_1,B_2]=-2JA^2+[B_1,B_2].\qedhere\]
\end{proof}

We will study the abelian case first. We will distinguish two cases, according to whether $A$ is semisimple; we will also assume that $A$ is nonzero, since otherwise the extension is trivial.

\begin{lemma}
\label{lemma:semisimple}
Let $(g,J)$ be a pseudo-K\"ahler structure on $\R^4$. Let $A$, $B_1$, $B_2$ be as in Lemma~\ref{lemma:AB1B2}. Assume that $A$ is nonzero and semisimple. Then there is a basis $e_1,\dotsc, e_4$ of $\R^4$ such that
\begin{gather*}
g=e^1\otimes e^1-e^2\otimes e^2+e^3\otimes e^3-e^4\otimes e^4, \quad Je_1=e_3, Je_2=e_4,\\
A=\begin{pmatrix}
a& a & 0 & 0 \\
-a & a & 0 & 0 \\
0 & 0 & -a & -a \\
0 & 0 & a & -a
\end{pmatrix}, \quad
B_1=\begin{pmatrix}
0 & k_1 & - k_2 & 0 \\
k_1& 0  & 0 & k_2\\
k_2 & 0 & 0 & k_1\\
0 & - k_2 & k_1 & 0
\end{pmatrix}, \quad
B_2=\begin{pmatrix}
0 & k_2 &  k_1 & 0 \\
k_2& 0  & 0 & -k_1\\
-k_1 & 0 & 0 & k_2\\
0 &  k_1 & k_2 & 0
\end{pmatrix},
\end{gather*}
where $k_1,k_2,a$ are real numbers such that $k_1^2+k_2^2=2a^2$.
\end{lemma}

\begin{proof}
We first show that there is an $A$-invariant, $g$-orthogonal decomposition
\begin{equation}
\label{eqn:Ainvariantspaces}
\R^4=V_+\oplus V_-,  \quad J(V_+)=V_-, \quad \dim V_\pm=2.
\end{equation}
Indeed, denote by $V_\lambda$ the eigenspaces of $A$ over $\R$, and by $W_\mu$ its eigenspaces over $\C$.
Since $A$ is symmetric, we have an orthogonal decomposition
\[\R^4=\bigoplus_{\lambda} V_\lambda \oplus \bigoplus_{\mu} [\![W_\mu]\!],\]
where $\lambda$ ranges among real eigenvalues and $\mu$ among nonreal eigenvalues, the notation $[\![W_\mu]\!]$ representing the space of real vectors in $W_\mu\oplus W_{\overline\mu}$. Since $J$ anticommutes with $A$, it maps each $V_\lambda$ to $V_{-\lambda}$ and $[\![W_\mu]\!]$ to $[\![W_{-\mu}]\!]$.\medskip

If all eigenvalues of $A$ are real, we can obtain \eqref{eqn:Ainvariantspaces} by fixing an eigenvector $v$ and an eigenvector $w$ not contained in $\Span{v,Jv}$, then setting $V_+=\Span{v,w}$.\medskip

If $A$ has a purely imaginary eigenvalue $\mu\neq0$, then its restriction to $[\![W_\mu]\!]$ is $g$-symmetric, but not diagonalizable; by the spectral theorem, this implies that the restriction of $g$ to $[\![W_{\mu}]\!]$ is indefinite. Since $[\![W_{\mu}]\!]$ is preserved by $J$, it has an indefinite $\mathrm{U}(p,q)$-structure defined by $g$ and $J$, which for dimensional reasons is only possible if $p=q=1$. Thus, the purely imaginary eigenvalue $\mu$ has multiplicity two. Since the conditions of Lemma~\ref{lemma:semisimple} imply that $JA$ is $g$-symmetric, we can define a scalar product
\[(v,w)\mapsto\langle v,w\rangle= g(v, JAw),\]
which cannot be definite, for otherwise the $\langle,\rangle$-symmetric operator $J$ would have real eigenvalues by the spectral theorem. Thus, there exists a nonzero $v$ which is lightlike relative to $\langle,\rangle$, i.e.\ $v$ is $g$-orthogonal to $JAv$. Then we obtain the splitting \eqref{eqn:Ainvariantspaces} by setting
\[V_+=\Span{v,Av}, \quad V_-=\Span{Jv,JAv};\]
the spaces $V_\pm$ are $2$-dimensional because $A$ has no real eigenvalues, and they intersect trivially because $A$ and $J$ anticommute.\medskip

On the other hand, if  an eigenvalue $\mu$ is neither real nor imaginary, then by a dimension count $\mu$ has multiplicity one and $\R^4=[\![W_{-\mu}]\!]\oplus [\![W_{\mu}]\!]$.\medskip

Using \eqref{eqn:Ainvariantspaces} and the fact that $A$ anticommutes with $J$, we can therefore assume that $A$ takes the block form
\[A=\begin{pmatrix} D & 0 \\ 0 & -D \end{pmatrix}\]
relative to a basis such that
\begin{equation}
\label{eqn:basisR4}
Je_1=e_3, Je_2=e_4,\quad  g=e^1\otimes e^1+e^3\otimes e^3 \pm (e^2\otimes
e^2+e^4\otimes e^4).
\end{equation}
Similarly, in the basis satisfying \eqref{eqn:basisR4}, write
\[B_i=\begin{pmatrix} H_i & -U_i \\ U_i & H_i\end{pmatrix}, \quad H_i^*=-H_i, U_i^*=U_i,\]
where  $*$ denotes the metric transpose taken relative to the  metric
\begin{equation}
\label{eqn:metriconR2}
\begin{pmatrix} 1 & 0 \\ 0 & \pm 1\end{pmatrix}.
\end{equation}
Then
\[[A,B_i]=\begin{pmatrix}[D,H_i] & -DU_i-U_iD\\
-DU_i-U_iD  & -[D,H_i]
\end{pmatrix}.
\]
Now \eqref{eqn:symmetricpartofabeliancondition} implies
\[
0=[A,B_2]-J[A,B_1]
=\begin{pmatrix}[D,H_2] & -DU_2-U_2D\\
-DU_2-U_2D  & -[D,H_2]
\end{pmatrix}
-\begin{pmatrix}
DU_1+U_1D  & [D,H_1]\\
[D,H_1] & -DU_1-U_1D\\
\end{pmatrix}.
\]
Therefore
\begin{equation}
\label{eqn:DKsymmetric}
[D,H_2]=DU_1+U_1D , \quad [D,H_1]=-DU_2-U_2D.
\end{equation}
If we set
\[K_1=H_1-U_2, K_2=H_2+U_1,\]
using \eqref{eqn:DKsymmetric} and $K_1^*=-H_1-U_2$, $K_2^*=-H_2+U_1$ we see that
\[K_2D=-DK_2^*, K_1D=-DK_1^*,\]
i.e.\ the $K_iD$ are skew-symmetric.\medskip

Now \eqref{eqn:skewsymmetricpartofabeliancondition} implies that
\[
0=-2JA^2+[B_1,B_2]
=\begin{pmatrix}
0 & 2D^2\\
-2D^2 & 0 \\
\end{pmatrix} +
\begin{pmatrix} [H_1,H_2]-[U_1,U_2] &-[H_1,U_2]+[H_2,U_1] \\
-[H_2,U_1]+[H_1,U_2] & -[U_1,U_2]+[H_1,H_2]\end{pmatrix}.
\]
Thus
\[[H_1,H_2]=[U_1,U_2], \quad [H_1,U_2]-[H_2,U_1]=2D^2,\]
which implies
\begin{equation}
\label{eqn:DandKstar}
[K_1,K_1^*]+[K_2,K_2^*]=-4D^2.
\end{equation}

Thus, $D^2$ has trace zero; as $D$ is semisimple, say with eigenvalues $\lambda_1,\lambda_2$ over $\C$, we see that $\lambda_1^2+\lambda_2^2$ is zero. Since $D$ is a real matrix, this implies that $\lambda_1$ and $\lambda_2$ are complex conjugate, so the only possibility is that $\lambda_2^2=-\lambda_1^2\in i\R$. Thus, $D$ represents a homothety composed with a rotation of angle $\frac\pi2$, i.e.\
\[D=\begin{pmatrix}a& a \\-a & a \end{pmatrix}, \quad a\in\R.\]
In particular, a minus sign appears in the metric \eqref{eqn:metriconR2}, as $D$ is symmetric relative to it. Antisymmetric matrices relative to \eqref{eqn:metriconR2} are multiples of
\[\begin{pmatrix}0&1 \\  1 &0\end{pmatrix};\]
multiplying by $D^{-1}$, we see that $K_iD$ being antisymmetric relative to \eqref{eqn:metriconR2} implies
\[K_i=\begin{pmatrix}k_i& k_i \\  k_i & -k_i\end{pmatrix}, k_i\in\R,
\]
i.e.\
\[U_1=\begin{pmatrix}k_2 & 0 \\ 0 & -k_2\end{pmatrix}, \quad
U_2=\begin{pmatrix}-k_1 & 0 \\ 0 & k_1\end{pmatrix}, \quad H_1=\begin{pmatrix}0 & k_1 \\ k_1&0\end{pmatrix}, \quad H_2=\begin{pmatrix}0 & k_2 \\ k_2&0\end{pmatrix}.\]

By \eqref{eqn:DandKstar}, we see that $k_1^2+k_2^2=2a^2$. Observe that $H_1,H_2$ are linearly dependent, as are $U_1,U_2$, so all conditions are satisfied.
\end{proof}

\begin{proposition}\label{prop:definite_yields_trivial_extension}
Any pseudo-K\"ahler extension of a $4$-dimensional definite K\"ahler Lie algebra $(\lie g,g,J_g)$ by an abelian two-dimensional pseudo-K\"ahler Lie algebra $(\lie h,h,J_h)$ is trivial.
\end{proposition}

\begin{proof}
Fix a linear isomorphism $\g\cong\R^4$; then $g$ and $J_g$ define a K\"ahler structure on the abelian Lie algebra $(\R^4,g,J_g)$, where $g$ and $J_g$ are induced by the isomorphism. A pseudo-K\"ahler extension of $(\lie g,g,J_g)$ by $(\lie h,h,J_h)$ determines a pseudo-K\"ahler extension of $(\R^4,g,J_g)$. Define $A,B_1,B_2$ as in Lemma~\ref{lemma:AB1B2}. If $A$ is not zero, then it is diagonalizable by the spectral theorem; applying Lemma~\ref{lemma:semisimple} to $(\R^4,g,J_g)$ we obtain a contradiction, because the metric $g$ is definite. Thus, $A$ is necessarily zero, i.e.\ the extension is trivial.
\end{proof}

\begin{lemma}
\label{lemma:nonsemisimple}
Let $(g,J)$ be a pseudo-K\"ahler structure on $\R^4$. Let $A$, $B_1$, $B_2$ be as in Lemma~\ref{lemma:AB1B2}. Assume that $A$ is not semisimple. Then there is a basis $e_1,\dotsc, e_4$ of $\R^4$ such that
\begin{gather*}
  g=e^1\otimes e^1+e^2\otimes e^2-e^3\otimes e^3-e^4\otimes e^4,
\quad Je_1=e_2, \quad Je_3=e_4,\\
A=\begin{pmatrix}a &0 &-a &0 \\ 0 &-a &0 &a \\ a &0 &-a&0 \\ 0 &-a & 0& a\end{pmatrix}, \quad
B_1=\begin{pmatrix}
0&\nu_2-\mu_1 & \mu_2&\mu_1 \\
-\nu_2+\mu_1 & 0 & -\mu_1 & \mu_2 \\
\mu_2 & -\mu_1 & 0 & \nu_2+\mu_1\\
\mu_1 & \mu_2 &-\nu_2-\mu_1 & 0
\end{pmatrix},\\
B_2=\begin{pmatrix}
0&-\mu_2-\nu_1 & \nu_2&\nu_1 \\
\mu_2+\nu_1 & 0 & -\nu_1 & \nu_2 \\
\nu_2 & -\nu_1 & 0 & -\mu_2+\nu_1\\
\nu_1 & \nu_2 &\mu_2-\nu_1 & 0
\end{pmatrix},
\end{gather*}
where $a\neq0$ and $(\nu_1,\nu_2),(\mu_1,\mu_2)\in\R^2$ are linearly dependent.
\end{lemma}

\begin{proof}
Since $A$ is not semisimple, there is an eigenvalue $\lambda$ whose generalized eigenspace is not spanned by eigenvectors; thus, it has dimension at least two. Because $A$ anticommutes with $J$, $-\lambda$ and $\overline\lambda$ also have generalized eigenspaces not spanned by eigenvectors, and for dimensional reasons $\lambda$ is either real or purely imaginary.\medskip

We claim that $\lambda$ is necessarily zero. Indeed, suppose that $\lambda=i$; then $A$ leaves invariant a $2$-dimensional space $[\![W_i]\!]$, on which it acts as a complex structure. Since $J$ anticommutes with $A$, it maps $[\![W_i]\!]$ to another $A$-invariant space, which is necessarily $[\![W_i]\!]$ itself. This would give two anticommuting complex structures on the $2$-dimensional space $[\![W_i]\!]$, which is impossible. This argument also shows that $\lambda$ cannot be a nonzero multiple of $i$.\medskip

Suppose now that $\lambda$ is a nonzero real number. Then $-\lambda$ is also an eigenvalue, and we obtain a decomposition into generalized eigenspaces
\[\R^4=V_\lambda\oplus V_{-\lambda};\]
the decomposition is orthogonal because $A$ is $g$-symmetric, and in particular the metric restricted to $V_\lambda$ is non-degenerate. If we consider the $g$-symmetric nilpotent endomorphism $A-\lambda I$ restricted to $V_\lambda$, we see that its kernel must be orthogonal to the image, which is again the kernel; this implies that the kernel is isotropic. In a suitably normalized Jordan basis, the restrictions of $A$ and $g$  take the form
\[A|_{V_{\lambda}}=\begin{pmatrix} \lambda & 1  \\ 0 & \lambda  \end{pmatrix},\quad g|_{V_\lambda}=\pm \begin{pmatrix} 0 & 1 \\ 1 & 0 \end{pmatrix}.\]
Using the fact that $J$ anticommutes with $A$, we obtain
\[A=\begin{pmatrix} \lambda & 1 & 0 & 0 \\ 0 & \lambda & 0 & 0 \\ 0 & 0 & -\lambda& -1 \\ 0 & 0 & 0 & -\lambda\end{pmatrix}, \quad g=\pm(e^1\odot e^2+e^3\odot e^4), \quad Je_1=e_3, Je_2=e_4.\]
Using the fact that the $B_i$ are $g$-antisymmetric and commute with $J$, one can write
\[B_i=\begin{pmatrix}a_i & 0 & b_i & c_i\\ 0 & -a_i& d_i & b_i \\ -b_i & -c_i & a_i & 0 \\ -d_i & -b_i & 0 & -a_i\end{pmatrix}.\]
One then sees that \eqref{eqn:skewsymmetricpartofabeliancondition} has no solution.\medskip

Therefore, the only possibility is that  $A$ only has  zero as an eigenvalue. We are assuming $A\neq0$;  since $A$ anticommutes with $J$, the kernel is $2$-dimensional and coincides with the image. Thus, $\ker A$ is isotropic and $J$-invariant. If we fix an orthonormal basis $\{e_i\}$ of $\R^4$ with
\[Je_1=e_2, \quad Je_3=e_4, \quad g=e^1\otimes e^1+e^2\otimes e^2-e^3\otimes e^3-e^4\otimes e^4,\]
we can replace $A$ with any matrix in the same orbit for the action of the group $\mathrm{U}(1,1)$ that preserves $g$ and $J$. Any lightlike vector takes the form
\[v+w, \quad v\in\Span{e_1,e_2}, w\in\Span{e_3,e_4},\]
where $v,w$ have the same  Euclidean norm. Up to $\mathrm{U}(1)\times\mathrm{U}(1)\subset \mathrm{U}(1,1)$, we can assume $v+w$ is a multiple of $e_1+e_3$. This shows that we can assume that the isotropic, $J$-invariant subspace $\ker A$ is spanned by $e_1+e_3, e_2+e_4$. Now consider the map
\[\Span{e_1,e_2}\xrightarrow{A}\Span{e_1+e_3,e_2+e_4}.\]
The matrix of this map relative to the natural bases is symmetric (as $A$ is $g$-symmetric), so it can be diagonalized. This means that we can use the action of the diagonal $\mathrm{U}(1)$ in $\mathrm{U}(1)\times\mathrm{U}(1)$ to assume that $A(e_1)$ is a multiple of $e_1+e_3$; up to scaling, this fully determines $A$ by $J$-antiinvariance.\medskip

To simplify the proof, we will rescale $A$ to obtain
\[A=\begin{pmatrix}1 &0 &-1 &0 \\ 0 &-1 &0 &1 \\ 1 &0 &-1 &0 \\ 0 &-1 & 0& 1\end{pmatrix}.\]
The generic element of $\lie u(1,1)$ takes the form
\[\begin{pmatrix}
0&-h-\mu_1 & \mu_2&\mu_3 \\
h+\mu_1 & 0 & -\mu_3 & \mu_2 \\
\mu_2 & -\mu_3 & 0 & -h+\mu_1\\
\mu_3 & \mu_2 & h-\mu_1 & 0
\end{pmatrix}.\]
Identifying a $g$-antisymmetric endomorphism $f$ with the $2$-form $g(f(\cdot),\cdot)$, the center is spanned by the matrix corresponding to the $2$-form $e^{12}+e^{34}$, and the complement is $\su(1,1)\cong\Sl(2,\R)$, which has rank one.\medskip

Since $A^2=0$, \eqref{eqn:skewsymmetricpartofabeliancondition} requires that $B_1$ and $B_2$ commute. Thus, their components in $\su(1,1)$ are linearly dependent. We obtain
\[B_1=h(e^{12}+e^{34})+x\beta, \quad B_2=k(e^{12}+e^{34})+y\beta,\]
where $\beta$ is some nonzero element in $\su(1,1)$, $h,k,x,y\in\R$ and we are identifying the skew-symmetric matrices $B_1,B_2$ with $2$-forms. Then \eqref{eqn:symmetricpartofabeliancondition} gives
\[[A, k(e^{12}+ e^{34})+y\beta]=J[A,h(e^{12}+e^{34})+x\beta].\]
Solving this equation in $h,k,x,y,\beta$ shows that $B_1$ and $B_2$ take the form in the statement.
\end{proof}

We are now in a position to classify non-trivial pseudo-K\"ahler extensions $\R^4\rtimes_\varphi\R^2$, where $\R^2$ has the positive-definite K\"ahler structure \eqref{eqn:kahleronR2}. Having established in Proposition~\ref{prop:definite_yields_trivial_extension} that definite metrics on $\R^4$ only yield trivial extensions, we will consider on $\R^4$ the neutral pseudo-K\"ahler structure given by
\begin{equation}
\label{eqn:pseudokahlerR4}
g=e^1\otimes e^1-e^2\otimes e^2+e^3\otimes e^3-e^4\otimes e^4, \quad Je_1=e_3, Je_2=e_4.
\end{equation}

\begin{proposition}
\label{prop:extendabelian}
Every non-trivial pseudo-K\"ahler extension of the form $\R^4\rtimes_\varphi\R^2$ is AW-related to one obtained by taking the pseudo-K\"ahler structure \eqref{eqn:pseudokahlerR4} on $\R^4$ and either
\begin{gather*}
\varphi(a_1)=
a\begin{pmatrix}
1& 2& -1& 0 \\
0& 1 & 0 & 1 \\
1& 0 & -1 & 0\\
0 & -1 & 2 & -1
\end{pmatrix},
\quad \varphi(a_2)=
a\begin{pmatrix}
0& 1& 2& 1 \\
1& 0 & -1& 0 \\
0& 1 & 0 & 1\\
-1 & 2 & 1 & 0
\end{pmatrix},
\end{gather*}
where $a>0$, or
\begin{gather*}
\varphi(a_1)=
\begin{pmatrix} a & -a+c & -b & b \\
a+c & -a & -b & b \\
b&-b& -a& a+c\\
b&-b& -a+c & a \end{pmatrix},
\quad
\varphi(a_2)=
\begin{pmatrix}
    0&0&a-c&-a\\
    0&0&a&-a-c\\
    a+c&-a&0&0\\
    a&-a+c&0&0
\end{pmatrix},
\end{gather*} where $a>0$ and either $b=0=c$ or $b,c\in\R$ with $c\neq0$.
\end{proposition}

\begin{proof}
Write $\varphi(a_1)=A+B_1$, $\varphi(a_2)=JA+B_2$ as in Lemma~\ref{lemma:AB1B2}. Consider the one-parameter group of automorphisms $\{\exp \theta J\}$, which preserves the pseudo-K\"ahler structure of $\R^4$. We see that
\begin{equation}
\label{eqn:AdJA}
\Ad(\exp\theta J)A=\cos 2\theta A + \sin 2\theta JA.
\end{equation}
In particular this shows that for every element $v\in \R^2$ of norm one, $\varphi(v)^s$ is conjugated to $A$. Thus, $A$ is semisimple if and only if $\varphi(v)^s$ is semisimple for all $v\in\R^2$. This condition also holds if one takes an AW-related extension, which amounts to acting by automorphisms and modifying the skew-symmetric part of $\varphi$.\medskip

Suppose first that $A$ is semisimple, so that $A,B_1,B_2$ take the form of Lemma~\ref{lemma:semisimple}. Since $B_1$ and $B_2$ are invariant under $\{\exp \theta J\}$ and in light of \eqref{eqn:AdJA}, the automorphisms $f_h=\exp (-\theta/2 J)$ and $f_g=\exp\theta J_h$ yield an AW-related extension with
\[\varphi(a_1)=A+(\cos\theta B_1+\sin\theta B_2), \quad \varphi(a_2)=JA+(-\sin\theta B_1+\cos\theta B_2).\]
In terms of the parameters $(a,k_1,k_2)$, automorphisms give a circle symmetry in the $(k_1,k_2)$-plane. Since $2a^2=k^2_1+k^2_2$, we can assume $k_1=k_2=a$. In addition, we may reverse the sign of $a$ by reflecting $\R^4$ around the plane $\Span{e_2,e_4}$.  Different choices of $a>0$ yield extensions which are not AW-related, because $\det A=(2a^2)^2$ is an invariant.\medskip

Now consider the case in which $A$ is not semisimple, and let $A,B_1,B_2$ be as in Lemma~\ref{lemma:nonsemisimple}. Arguing as above, we see that we can rotate in the $(B_1,B_2)$-plane, which means that we can fix any angle $\theta$ and obtain new parameters
\[\begin{pmatrix}
    \mu'_1 & \nu'_1\\
    \mu'_2  & \nu'_2
\end{pmatrix}
=
\begin{pmatrix}
    \mu_1 & \nu_1\\
    \mu_2  & \nu_2
\end{pmatrix}
\begin{pmatrix}
    \cos\theta & -\sin\theta\\
    \sin\theta & \cos\theta
\end{pmatrix}.
\]
Since $\big(\begin{smallmatrix}
    \mu_1 & \nu_1\\
    \mu_2  & \nu_2
\end{smallmatrix}\big)$ is not invertible, we may assume that $\nu_1=0=\nu_2$. Changing the basis so that the metric takes the form \eqref{eqn:pseudokahlerR4}, we obtain
\[
A=\begin{pmatrix} a & -a & 0 & 0 \\
a & -a & 0 & 0 \\
0&0& -a& a\\
0&0& -a & a \end{pmatrix}, \quad B_1=\begin{pmatrix}
    0 & c & -b &    b\\
    c & 0  & -b & b \\
    b & -b & 0 & c \\
    b & -b & c & 0
\end{pmatrix}, \quad
B_2=\begin{pmatrix}
    0 &  0 & -c & 0\\
0&0&0& -c \\
    c & 0 & 0 & 0\\
    0&c & 0 & 0
\end{pmatrix}.
\]
Now suppose that there are two different AW-related extensions of this form, $A,B_1,B_2$, $A',B_1',B_2'$. We can assume $A=A'$. By Definition~\ref{def:equivalence}, $B_1'-B_1$ and $B_2'-B_2$ commute with both $A+B_1$ and $JA+B_2$. Separating the $g$-symmetric and $g$-antisymmetric part, this boils down to $B_1'-B_1$, $B_2'-B_2$ commuting with each of $A$, $B_1$ and $B_2$. This happens  if and only if $c=0=c'$ with $b,b'$ arbitrary; this shows that when $b$ can be set to zero if $c=0$, proving that $\varphi$ has the form in the statement up to isometry.
\end{proof}

Outside of the abelian case, it turns out that there are only two pseudo-K\"ahler Lie algebras admitting a non-trivial extension. They are the Lie algebra
$\frg=\mathfrak{rh}_3$, with nonzero Lie brackets
\[[e_1,e_2]=e_3,\]
and the pseudo-K\"ahler structure given by
\begin{equation}
\label{eqn:kahleronrh3}
Je_1=e_2, Je_3=e_4,  \quad g=e^1\odot e^3+e^2\odot e^4,
\end{equation}
and $\frg=\mathfrak{r}'_2$, whose nonzero Lie brackets are $$[e_1,e_3]=e_3,\quad[e_1,e_4]=e_4,\quad[e_2,e_3]=e_4,\quad[e_2,e_4]=-e_3,$$
and the pseudo-K\"ahler structure given by
\begin{equation}
\label{eqn:kahleronr2}
J=\left(\begin{array}{rrrr}
0 & 1 & 0 & 0 \\
-1 & 0 & 0 & 0 \\
0 & 0 & 0 & -1 \\
0 & 0 & 1 & 0
\end{array}\right)\quad\text{and}\quad g=\left(\begin{array}{rrrr}
a_{12} & 0 & -a_{14} & a_{13} \\
0 & a_{12} & a_{13} & a_{14} \\
-a_{14} & a_{13} & 0 & 0 \\
a_{13} & a_{14} & 0 & 0
\end{array}\right), \quad a_{13}^2+a_{14}^2\neq0.
\end{equation}

\begin{remark}
In \cite{Ova06}, $\mathfrak{rh}_3$ appears with a two-parameter family of pseudo-K\"ahler structures. We are indebted to Federico A.\ Rossi for pointing out to us that they are related to each other by isometric isomorphisms. On the other hand, $\mathfrak{r}'_2$ admits a distinct pseudo-K\"ahler structure, which does not admit non-trivial extensions.
\end{remark}

\begin{theorem}
\label{thm:classification42}
Every non-trivial extension of a $4$-dimensional pseudo-K\"ahler Lie algebra by a $2$-dimensional abelian Lie algebra is AW-related to one of:
\begin{itemize}
\itemsep 0em
\item the extensions of $\R^4$ given in Proposition~\ref{prop:extendabelian};
\item the extension of $\mathfrak{rh}_3$ with the pseudo-K\"ahler structure \eqref{eqn:kahleronrh3} and
\[
\varphi(a_1)=
\begin{pmatrix}
0 & 0 & 0 & 0 \\
0 & 0 & 0 & 0 \\
a & 0 & 0 & 0 \\
0 & -a & 0 & 0
\end{pmatrix},
\quad \varphi(a_2)=
\begin{pmatrix}
0 & 0 & 0 & 0 \\
0 & 0 & 0 & 0 \\
0 & a & 0 & 0 \\
a & 0 & 0 & 0
\end{pmatrix}, \quad a>0;
\]
\item the extension of $\mathfrak{r}'_2$ with the pseudo-K\"ahler structure \eqref{eqn:kahleronr2} and
$$\varphi(a_1)=\left(\begin{array}{rrrr}
0 & 0 & 0 & 0 \\
0 & 0 & 0 & 0 \\
a & 0 & 0 & 0 \\
0 & a & 0 & 0
\end{array}\right),\quad\varphi(a_2)=\left(\begin{array}{rrrr}
0 & 0 & 0 & 0 \\
0 & 0 & 0 & 0 \\
0 & -a & 0 & 0 \\
a & 0 & 0 & 0
\end{array}\right), \quad a>0.$$
\end{itemize}
\end{theorem}

\begin{proof}
We illustrate the computation for $\lie{rh}_3$. In this case derivations take the block form
\begin{equation}
\label{eqn:blockform}
D=\begin{pmatrix} D_t & 0 \\D_{l} & D_{r}\end{pmatrix},
\end{equation}
where $D_t$, $D_{l}$ and $D_{r}$ are two by two matrices, with the subscript indicating the position of the corresponding block, and $D_r$ is upper triangular. Relative to the metric $g$ defined in \eqref{eqn:kahleronrh3}, the $g$-symmetric and $g$-antisymmetric part of a derivation also has the block form \eqref{eqn:blockform}. If $A,B_1,B_2$ are as in Lemma~\ref{lemma:AB1B2}, then the blocks $(B_1)_t$, $(B_2)_t$ are in the centralizer of the complex structure $e^1\otimes e_2-e^2\otimes e_1$, so that \[[B_1,B_2]_t=[(B_1)_t,(B_2)_t]=0.\] Then \eqref{eqn:skewsymmetricpartofabeliancondition} implies that $(A^2)_t=0$. Since the block $A_t$ anticommutes with the complex structure, it is symmetric and trace-free, which makes $A_t$ zero, as it is a symmetric matrix that squares to zero. Then $g$-symmetry of $A$ implies that the block $A_r$ also vanishes. Keeping in mind that $A$ anticommutes with $J$, we see that the block $A_l$ is symmetric and trace-free. Up to an automorphism that rotates the basis $\{a_1,a_2\}$ of \eqref{eqn:kahleronR2}, and consequently $A$ and $JA$, we can assume that $A$ takes the form
\[A=\begin{pmatrix}0 & 0 & 0 & 0 \\
0 & 0 & 0 & 0 \\
a & 0 & 0 & 0 \\
0 & -a & 0 & 0
    \end{pmatrix},\quad a>0.\]
Since $A_r$ and $(JA)_r$ vanish, the blocks $(B_i)_r$ are upper triangular.
Imposing in addition that the $B_i$ are $g$-antisymmetric and commute with $J$, we find
\[B_i=
\left(
\begin{array}{cccc}
 h_i & 0 & 0 & 0 \\
 0 & h_i & 0 & 0 \\
 0 & k_i & -h_i&0 \\
 -k_i & 0 & 0 & -h_i \\
\end{array}
\right),\quad h_i,k_i\in\R.\]
Then \eqref{eqn:symmetricpartofabeliancondition} implies that $h_1=0=h_2$; it follows that the $B_i$ are derivations that commute with $A+B_1$, $JA+B_2$, and the pseudo-K\"ahler extension is AW-related to one with $k_1=0=k_2$.\medskip

For $\mathfrak{r}'_2$, the space of derivations is
\begin{equation}
\label{eqn:derr2}
\begin{pmatrix}
 0 & 0 & 0 & 0 \\
 0 & 0 & 0 & 0 \\
 a & -b & c & d \\
 b & a  & -d & c
\end{pmatrix}.
\end{equation}
To determine extensions relative to the pseudo-K\"ahler structure \eqref{eqn:kahleronr2}, observe that irrespective of the metric imposing that $A$ anticommute with $J$ and $B_i$ commute with $J$ forces
\[A=\begin{pmatrix}
 0 & 0 & 0 & 0 \\
 0 & 0 & 0 & 0 \\
 a & -b & 0 & 0 \\
 b & a  & 0 & 0
\end{pmatrix}, \quad B_i=
\begin{pmatrix}
 0 & 0 & 0 & 0 \\
 0 & 0 & 0 & 0 \\
 0 & 0 & c_i & d_i \\
 0 & 0  & -d_i & c_i
\end{pmatrix}.\]
Again, we can assume $a\geq0$, $b=0$ up to isometry, and the extension is only non-trivial if $a\neq0$. Imposing that $B_i$ is skew-symmetric gives $B_i=0$.\medskip

On the other hand, $\mathfrak{r}'_2$ also has a pseudo-K\"ahler structure
$$J'=\left(\begin{array}{rrrr}
0 & 0 & -1 & 0 \\
0 & 0 & 0 & -1 \\
1 & 0 & 0 & 0 \\
0 & 1 & 0 & 0
\end{array}\right)\quad\text{and}\quad g'=\left(\begin{array}{rrrr}
-a_{13} & -a_{14} & 0 & 0 \\
-a_{14} & a_{13} & 0 & 0 \\
0 & 0 & -a_{13} & -a_{14} \\
0 & 0 & -a_{14} & a_{13}
\end{array}\right),\quad  a_{13}^2+a_{14}^2\neq0.$$
Suppose $A,B_1,B_2$ satisfy Lemma~\ref{lemma:AB1B2}. The component of the derivation \eqref{eqn:derr2} that commutes with $J$ is $g$-antisymmetric if and only if $c=d=0$, so we must have
\[
A=
\begin{pmatrix}
0 & 0 & a & -b \\
 0 & 0 & b & a \\
 a & -b & 0 & 0 \\
 b & a & 0 & 0 \\
 \end{pmatrix},
 \quad B_1=
\begin{pmatrix}
 0 & 0 & -a & b \\
 0 & 0 & -b & -a \\
 a & -b & 0 & 0 \\
 b & a & 0 & 0 \\
\end{pmatrix},\]
and $B_2$ has the same form as $B_1$; however, $JA+B_2$ cannot be a derivation unless $a$ and $b$ are zero. This shows that every extension of this pseudo-K\"ahler structure is trivial.\medskip

The other cases are similar.
\end{proof}

We would like to thank Luis Ugarte for the discussion that led to the following remark.

\begin{remark}
    Denote by $\tilde{\frg}_{a,b,c}=\R^4\rtimes_\varphi\R^2$ the second family of pseudo-K\"ahler extensions appearing in Proposition~\ref{prop:extendabelian}, depending on parameters $a,b,c$; $\tilde{\frg}_{a,b,c}$ is nilpotent if and only if $c=0$. In fact, $\tilde{\frg}_{a,b,0}$ is $2$-step nilpotent by Remark~\ref{remark:2-step}. Pseudo-Kähler nilpotent Lie algebras of dimension $6$ were classified in \cite[Theorem~3.1]{CFU04}. In their notation, $\tilde{\frg}_{0,b,0}\cong\frh_6$ for any $b\in\R\setminus\{0\}$ and $$\tilde{\frg}_{a,b,0}\cong\begin{cases}
        \frh_5&\mbox{if }b^2<1,\\
        \frh_4&\mbox{if }b^2=1,\\
        \frh_2&\mbox{if }b^2>1,\\
    \end{cases}$$ for any $a>0$. The extension $\tilde{\frg}_a=\mathfrak{rh}_3\rtimes_{\varphi_a}\R^2$ is always $2$-step nilpotent with abelian complex structure. In the notation of \cite{CFU04}, $\tilde{\frg}_0\cong\frh_8$ and $\tilde{\frg}_a\cong\frh_5$ when $a\neq0$.
\end{remark}

We conclude this section by listing explicitly the $6$-dimensional pseudo-K\"ahler Lie algebras obtained in Theorem~\ref{thm:classification42}.
\begin{itemize}
\itemsep 0em
\item Extension of
$\frg$ abelian, $\varphi(a_i)$ semisimple: normalizing to $a=1$ (which has the effect of rescaling the metric), with notation as in Example~\ref{example:withthenotation}, we obtain the Lie algebra
\begin{multline*}
\big(e^{15}+2 e^{25}+e^{26}-e^{35}+2 e^{36}+e^{46},e^{16}+e^{25}-e^{36}+e^{45},\\
e^{15}+e^{26}-e^{35}+e^{46},-e^{16}-e^{25}+2 e^{26}+2 e^{35}+e^{36}-e^{45},0,0\big);
\end{multline*}
the metric and K\"ahler form are given by
\[\tilde g=e^1\otimes e^1-e^2\otimes e^2+e^3\otimes e^3-e^4\otimes e^4+e^5\otimes e^5+e^6\otimes e^6, \quad\tilde\omega=e^{12}-e^{34}+e^{56}.\]
This metric is Ricci-flat but not flat.

\item Extension of
$\frg$ abelian, $\varphi(a_i)$ not semisimple: rescaling the metric so that $a=1$, we obtain
\begin{multline*}
\big(e^{15}- e^{25}+c e^{25}-b e^{35}+ e^{36}-c e^{36}+b e^{45}- e^{46},-(1+c) (-e^{15}+e^{46})+ (-e^{25}+e^{36})+b (-e^{35}+e^{45}),\\
(1+c) (e^{16}+e^{45})- (e^{26}+e^{35})-b (-e^{15}+e^{25}),-(1-c) (e^{26}+e^{35})+ (e^{16}+e^{45})-b (-e^{15}+e^{25}),0,0\big)
\end{multline*}
with metric and K\"ahler form
\[\tilde g=e^1\otimes e^1-e^2\otimes e^2+e^3\otimes e^3-e^4\otimes e^4+e^5\otimes e^5+e^6\otimes e^6, \quad\tilde\omega=e^{13}-e^{24}+e^{56},\]
also Ricci-flat but not flat.
\item Extension of $\lie{rh}_3$:
\[(0,0,a (e^{15}+e^{26})-e^{12},-a (-e^{16}+e^{25}),0,0)\]
with metric and K\"ahler form
\[\tilde g=e^1\odot e^3+e^2\odot e^4+e^5\otimes e^5+e^6\otimes e^6, \quad\tilde\omega=e^{14}-e^{23}+e^{56},\]
also Ricci-flat but not flat, unless $a=0$.
\item Extension of
$\lie{r}'_2$:
\[\big(0,0,-a (-e^{15}+e^{26})-e^{13}+e^{24},a (e^{16}+e^{25})-e^{14}-e^{23},0,0\big)\]
with metric and K\"ahler form \begin{align*}
    \tilde g&=x(e^1\otimes e^1+e^2\otimes e^2)+y(e^2\odot e^4-e^1\odot e^3)+z(e^1\odot e^4+e^2\odot e^3)+e^5\otimes e^5+e^6\otimes e^6,\\
    \tilde\omega&=-xe^{12}-ze^{13}-ye^{14}-ye^{23}+ze^{24}+e^{56},
\end{align*} where $(y,z)\neq(0,0)$. The metric is Ricci-flat for all $x,y,z$, and flat precisely when $x=a^2 (y^2 + z^2)$.
\end{itemize}

\section{Hypersymplectic structures}

In this section we construct hypersymplectic structures on the semidirect product $\tilde\frg=\frg\rtimes_\varphi\frh$ equipped with the pseudo-K\"ahler structure given by Theorem~\ref{thm:construction}. We use this method to produce new examples of hypersymplectic Lie algebras.\medskip

First of all let us set the following definitions.

\begin{definition}
    We say that $(\frg,g,J,E)$ is an \emph{almost hypersymplectic Lie algebra} if $(g,J)$ is a pseudo-K\"ahler structure and $E\in\End(\frg)$ is an almost para-complex structure (i.e.\ $E^2=\mathbbm{1}$) such that $JE=-EJ$ and $g(E\cdot,E\cdot)=-g$. If in addition $E$ is parallel with respect to the Levi-Civita connection of $g$, then $(\frg,g,J,E)$ is called a \emph{hypersymplectic Lie algebra}.
\end{definition}

Recall that the existence of an almost hypersymplectic structure on a Lie algebra implies that its dimension is a multiple of $4$ and the metric has neutral signature (see \cite[p.~2043]{And06}).

\begin{example}\label{ex:flat_HS_R4}
    The basic example of a hypersymplectic Lie algebra is $\R^4$ equipped with the flat metric $$g=e^1\odot e^4-e^2\odot e^3$$ and complex and para-complex structures given by \begin{align*}
        J&=e^1\otimes e_3+e^2\otimes e_4-e^3\otimes e_1-e^4\otimes e_2,\\
        E&=e^1\otimes e_1+e^2\otimes e_2-e^3\otimes e_3-e^4\otimes e_4.
    \end{align*}

    The space of endomorphisms that commute with both $J$ and $E$ as above is $$\left\{\begin{pmatrix}
        a & b & 0 & 0\\
        c & d & 0 & 0\\
        0 & 0 & a & b\\
        0 & 0 & c & d
    \end{pmatrix}\mid a,b,c,d\in\R\right\}.$$

    If we require that these endomorphisms are moreover $g$-antisymmetric, then we need the condition $a+d=0$. So we find the space $$\left\{\begin{pmatrix}
        A&0\\
        0&A
    \end{pmatrix}\mid A\in\lie{sl}(2,\R)\right\}.$$

    This gives us the inclusion $\lie{sl}(2,\R)\cong\lie{sp}(2,\R)\hookrightarrow\lie{gl}(4,\R)$. More generally, by considering the flat hypersymplectic structure on $\R^{4n}$, we get $\lie{sp}(2n,\R)\hookrightarrow\lie{gl}(4n,\R)$.
\end{example}

Recall from \cite[p.~255]{FPPS04} that a \emph{Kodaira manifold} is a compact complex manifold $(M,J)$, where $M$ is a nilmanifold whose universal cover is 2-step nilpotent and has half-dimensional center, and $J$ is a left-invariant complex structure preserving the center of the corresponding Lie algebra. Inspired by this notion of Kodaira manifold we consider the following class of hypersymplectic Lie algebras.

\begin{definition}\label{def:Kodaira_type}
    Let $(\frg,g,J,E)$ be a hypersymplectic Lie algebra. We say that it is of \emph{Kodaira type} if $\frg$ is $2$-step nilpotent and its center is half-dimensional and $J$-invariant.
\end{definition}

The only non-abelian $4$-dimensional nilpotent hypersymplectic Lie algebra is of Kodaira type (see \cite[Theorem~23]{And06}). In \cite[Section~3.2]{FPPS04} some examples of hypersymplectic Lie algebras of Kodaira type are constructed in dimension $8$. In \cite[Section~5.1]{AD06} examples of hypersymplectic Lie algebras of Kodaira type are constructed in dimension $8n$ for $n\geq1$ and in \cite[Section~7.3]{CPO11} the authors construct examples in dimension $4n$ for $n\geq1$. It is pointed out in \cite[p.~261]{FPPS04} that the metric of a hypersymplectic Lie algebra of Kodaira type is flat. Let us state it for future reference.

\begin{proposition}[Fino-Pedersen-Poon-Sørensen]\label{prop:Kodaira_type_flat}
    Let $(\frg,g,J,E)$ be a hypersymplectic Lie algebra of Kodaira type. Then the metric $g$ is flat.
\end{proposition}

There is another recurring type of hypersymplectic structure for which it will be useful to have a definition. First let us recall that a complex structure $J$ is called \emph{abelian} if $$[JX,JY]=[X,Y].$$

Similarly, a para-complex structure $E$ will be called \emph{abelian} if $$[EX,EY]=-[X,Y].$$

It was shown in \cite[Proposition~6.1]{AD06} that if $(\frg,g,J,E)$ is hypersymplectic, then $J$ is abelian if and only if $E$ is abelian. Therefore, we give the following definition.

\begin{definition}
    Let $(\frg,g,J,E)$ be a hypersymplectic Lie algebra. We say that the hypersymplectic structure is \emph{abelian} if $J$ or $E$ is abelian. We will then refer to $(\frg,g,J,E)$ as a hypersymplectic Lie algebra of \emph{abelian type}.
\end{definition}

Although this nomenclature is not standard, we use it to distinguish a hypersymplectic Lie algebra of abelian type, which does not necessarily have an underlying abelian Lie algebra, from the abelian hypersymplectic Lie algebra, which is $\R^{4n}$ equipped with the canonical hypersymplectic structure as in Example~\ref{ex:flat_HS_R4}.

\begin{remark}
\label{remark:structureofabeliancomplex}
    If a real Lie algebra admits an abelian complex structure $J$, then it must be $2$-step solvable (see \cite[p.~235]{ABD11}, \cite{Pet88}). Furthermore, its center is even-dimensional, as it is $J$-invariant.
\end{remark}

In \cite[Theorem~2]{BS18} the authors classify the $8$-dimensional hypersymplectic Lie algebras of abelian type and in \cite[p.~15]{AD06} it is pointed out that all the possible hypersymplectic Lie algebras of abelian type are constructed as in \cite[Theorem~2.1]{AD06}. This observation, together with \cite[Theorem~4.2]{AD06}, gives us the following.

\begin{proposition}[Andrada-Dotti]\label{prop:2-step_abelian_flat}
    Let $(\frg,g,J,E)$ be a hypersymplectic Lie algebra of abelian type. If $\frg$ is $2$-step nilpotent, then the metric $g$ is flat.
\end{proposition}

\begin{remark}
    This is not longer true if $\frg$ is $3$-step nilpotent (see \cite{AD06,BS18}).
\end{remark}

\begin{remark}
    In dimension 4, the only non-abelian nilpotent hypersymplectic Lie algebra is both of Kodaira type and of abelian type. However, in general hypersymplectic Lie algebras of Kodaira type are not a subclass of hypersymplectic Lie algebras of abelian type. Indeed, the $8$-dimensional hypersymplectic Lie algebra of Kodaira type constructed in \cite[Example~2]{FPPS04} is not of abelian type.
\end{remark}

\subsection{Construction and examples}

Now we can proceed with the construction of hypersymplectic structures. Given the semidirect product $\tilde\frg=\frg\rtimes_\varphi\frh$ with the pseudo-K\"ahler structure $(\tilde g,\tilde J)$ constructed in Theorem~\ref{thm:construction}, we are interested in finding an almost para-complex structure $\tilde E$ on $\tilde\frg$ which is parallel with respect to the Levi-Civita connection $\tilde\nabla$ of $\tilde g$ so that $(\tilde\frg,\tilde g,\tilde J,\tilde E)$ is a hypersymplectic Lie algebra. For this, let us consider an almost para-complex structure $$\tilde E=\begin{pmatrix}
    E_1&E_2\\
    E_3&E_4
\end{pmatrix}\in\End(\tilde\frg),$$ where $E_1\in\End(\frg)$, $E_2\in\Hom(\frh,\frg)$, $E_3\in\Hom(\frg,\frh)$ and $E_4\in\End(\frh)$. The next theorem gives us the conditions for such $\tilde E$ to be parallel.

\begin{theorem}\label{thm:general_HS}
    Let $(\frg,g,J_g)$ and $(\frh,h,J_h)$ be pseudo-K\"ahler Lie algebras and let $\varphi:\frh\to\Der(\frg)$ be a representation satisfying the conditions of Theorem~\ref{thm:construction}. Suppose that $(\tilde g,\tilde J,\tilde E)$ is an almost hypersymplectic structure on $\tilde\frg$. Then $(\tilde\frg,\tilde g,\tilde J,\tilde E)$ is a hypersymplectic Lie algebra if and only if \begin{itemize}
        \itemsep 0em
        \item $\nabla^hE_4=0$,
        \item $\varphi(A)^a\circ E_2=E_2\circ\nabla^h_A$,
        \item $\varphi(A)^a\circ E_1=E_1\circ\varphi(A)^a$,
        \item $\varphi(B)^sE_2C=\varphi(C)^sE_2B$,
        \item $g(\nabla^g_XY,E_2C)=g((E_1\circ\varphi(C)^s-\varphi(E_4C)^s)X,Y)$,
        \item $g((\nabla^g_XE_1)Y,Z)=g(\varphi(E_3Y)^sZ-\varphi(E_3Z)^sY,X)$.
    \end{itemize}
\end{theorem}

\begin{proof}
    We need to determine the conditions for $\tilde E$ being $\tilde\nabla$-parallel. We have to consider four cases depending on where the vectors $u,v\in\tilde\frg$ on $(\tilde\nabla_u\tilde E)v$ lie. We will make use of Lemma~\ref{lemma:tilde_nabla_XY} throughout the proof.\medskip

    \textbullet\textbf{Case $u,v\in\frh$:} Set $u=A,v=B$. We compute \begin{align*}
        (\tilde\nabla_A\tilde E)B&=\tilde\nabla_A(\tilde EB)-\tilde E(\tilde\nabla_AB)\\
        &=\tilde\nabla_A(E_2B)+\tilde\nabla_A(E_4B)-E_2(\tilde\nabla_AB)-E_4(\tilde\nabla_AB)\\
        &=\varphi(A)^aE_2B+\nabla^h_AE_4B-E_2(\nabla^h_AB)-E_4(\nabla^h_AB).
    \end{align*}

    This term is zero if and only if $\varphi(A)^a\circ E_2=E_2\circ\nabla^h_A$ for all $A\in\frh$ and $\nabla^hE_4=0$.\medskip

    \textbullet\textbf{Case $u\in\frh$ and $v\in\frg$:} Set $u=A,v=Y$. We compute \begin{align*}
        (\tilde\nabla_A\tilde E)Y&=\tilde\nabla_A(\tilde EY)-\tilde E(\tilde\nabla_AY)\\
        &=\tilde\nabla_A(E_1Y)+\tilde\nabla_A(E_3Y)-E_1(\tilde\nabla_AY)-E_3(\tilde\nabla_AY)\\
        &=\varphi(A)^aE_1Y+\nabla^h_AE_3Y-E_1(\varphi(A)^aY)-E_3(\varphi(A)^aY).
    \end{align*}

    This term is zero if and only if $\varphi(A)^a\circ E_1=E_1\circ\varphi(A)^a$ and $\nabla^h_A\circ E_3=E_3\circ\varphi(A)^a$ for all $A\in\frh$.\medskip

    The condition $\nabla^h_A\circ E_3=E_3\circ\varphi(A)^a$ is equivalent to $\varphi(A)^a\circ E_2=E_2\circ\nabla^h_A$. Indeed, $\tilde E^*=-\tilde E$ implies that $E_3=-h^{-1}E_2^\top g$. The claimed equivalence follows from this and the fact that $\varphi(A)^a$ is $g$-antisymmetric and $\nabla^h_A$ is $h$-antisymmetric.\medskip

    \textbullet\textbf{Case $u\in\frg$ and $v\in\frh$:} Set $u=X,v=B$. Let $Z\in\frg$ and $C\in\frh$. We compute $$\tilde g((\tilde\nabla_X\tilde E)B,Z+C)=\tilde g(\tilde\nabla_X(\tilde EB),Z+C)-\tilde g(\tilde E(\tilde\nabla_XB),Z+C),$$ where \begin{align*}
        \tilde g(\tilde\nabla_X(\tilde EB),Z+C)&=\tilde g(\tilde\nabla_XE_2B,Z+C)+\tilde g(\tilde\nabla_XE_4B,Z+C)\\
        &=g(\nabla^g_XE_2B,Z)+g(\varphi(C)^sX,E_2B)-g(\varphi(E_4B)^sX,Z),\\
        \tilde g(\tilde E(\tilde\nabla_XB),Z+C)&=-\tilde g(\tilde\nabla_XB,\tilde EZ+\tilde EC)\\
        &=\tilde g(\varphi(B)^sX,E_1Z+E_3Z+E_2C+E_4C)\\
        &=g(\varphi(B)^sX,E_1Z)+g(\varphi(B)^sX,E_2C).
    \end{align*}

    Hence we have \begin{align*}
        \tilde g((\tilde\nabla_X\tilde E)B,Z+C)&=g(\nabla^g_XE_2B,Z)+g(\varphi(C)^sX,E_2B)-g(\varphi(E_4B)^sX,Z)\\
        &\quad-g(\varphi(B)^sX,E_1Z)-g(\varphi(B)^sX,E_2C)\\
        &=g(\nabla^g_XE_2B,Z)+g((E_1\circ\varphi(B)^s-\varphi(E_4B)^s)X,Z)\\
        &\quad+g(\varphi(C)^sX,E_2B)-g(\varphi(B)^sX,E_2C).
    \end{align*}

    This term is zero if and only if \begin{align*}
        g(\nabla^g_XE_2B,Z)&=g((\varphi(E_4B)^s-E_1\circ\varphi(B)^s)X,Z),\\
        g(\varphi(C)^sX,E_2B)&=g(\varphi(B)^sX,E_2C).
    \end{align*} The second equation is equivalent to $\varphi(B)^sE_2C=\varphi(C)^sE_2B$ for all $B,C\in\frh$.\medskip

    \textbullet\textbf{Case $u,v\in\frg$:} Set $u=X,v=Y$. We compute $$\tilde g((\tilde\nabla_X\tilde E)Y,Z+C)=\tilde g(\tilde\nabla_X(\tilde EY),Z+C)-\tilde g(\tilde E(\tilde\nabla_XY),Z+C),$$ where \begin{align*}
        \tilde g(\tilde\nabla_X(\tilde EY),Z+C)&=\tilde g(\tilde\nabla_XE_1Y,Z+C)+\tilde g(\tilde\nabla_XE_3Y,Z+C)\\
        &=g(\nabla^g_XE_1Y,Z)+g(\varphi(C)^sX,E_1Y)-g(\varphi(E_3Y)^sX,Z),\\
        \tilde g(\tilde E(\tilde\nabla_XY),Z+C)&=-\tilde g(\tilde\nabla_XY,\tilde EZ+\tilde EC)\\
        &=-\tilde g(\tilde\nabla_XY,E_1Z+E_3Z+E_2C+E_4C)\\
        &=-g(\nabla_X^gY,E_1Z)-g(\nabla_X^gY,E_2C)\\
        &\quad-g(\varphi(E_3Z)^sX,Y)-g(\varphi(E_4C)^sX,Y).
    \end{align*}

    Hence we have \begin{align*}
        \tilde g((\tilde\nabla_X\tilde E)Y,Z+C)&=g(\nabla^g_XE_1Y,Z)+g(\varphi(C)^sX,E_1Y)-g(\varphi(E_3Y)^sX,Z)\\
        &\quad+g(\nabla_X^gY,E_1Z)+g(\nabla_X^gY,E_2C)\\
        &\quad+g(\varphi(E_3Z)^sX,Y)+g(\varphi(E_4C)^sX,Y)\\
        &=g(\nabla_X^gY,E_2C)+g((\varphi(E_4C)^s-E_1\circ\varphi(C)^s)X,Y)\\
        &\quad+g((\nabla^g_XE_1)Y,Z)-g(\varphi(E_3Y)^sX,Z)+g(\varphi(E_3Z)^sX,Y).
    \end{align*}

    This term is zero if and only if \begin{align*}
        g(\nabla^g_XY,E_2C)&=g((E_1\circ\varphi(C)^s-\varphi(E_4C)^s)X,Y),\\
        g((\nabla^g_XE_1)Y,Z)&=g(\varphi(E_3Y)^sX,Z)-g(\varphi(E_3Z)^sX,Y)\\
        &=g(\varphi(E_3Y)^sZ-\varphi(E_3Z)^sY,X).
    \end{align*}

    The first equation above is equivalent to the condition $$g(\nabla^g_XE_2B,Z)=g((\varphi(E_4B)^s-E_1\circ\varphi(B)^s)X,Z),$$ since $\nabla^g_X$ is $g$-antisymmetric.
\end{proof}

We consider a couple of particular cases in the following corollary.

\begin{corollary}\label{cor:HS}
    In the situation of Theorem~\ref{thm:general_HS}: \begin{enumerate}[\normalfont(1)]
        \itemsep 0em
        \item If $E_2=E_3=0$, then $(\frg,g,J_g,E_1)$ and $(\frh,h,J_h,E_4)$ are hypersymplectic Lie algebras and $\varphi(A)^s=0$ for all $A\in\frh$.
        \item If $E_1\neq0$, $E_4\neq0$ and $\frg$ is abelian, then $\varphi(A)^s=0$ for all $A\in\frh$.
    \end{enumerate}
\end{corollary}

\begin{proof}
    (1) In this situation, $(g,J_g,E_1)$ and $(h,J_h,E_4)$ are almost hypersymplectic structures, and since $\nabla^gE_1=0$ and $\nabla^hE_4=0$, then they define hypersymplectic structures on $\frg$ and $\frh$, respectively. Moreover, the condition $E_1\circ\varphi(A)^s=\varphi(E_4A)^s$, together with $E_1^*=-E_1$ and $J_gE_1=-E_1J_g$, implies that $\varphi(A)^s=0$ for all $A\in\frh$. Indeed, since $J_g\circ\varphi(A)^s=\varphi(J_hA)^s$ is $g$-symmetric, we have \begin{equation}\label{eqn:JanticommutesphiA}
        J_g\circ\varphi(A)^s=(J_g\circ \varphi(A)^s)^*=\varphi(A)^s\circ J_g^*=-\varphi(A)^s\circ J_g.
    \end{equation}

    Similarly, $E_1\circ\varphi(A)^s=-\varphi(A)^s\circ E_1.$ Hence, if $B=J_hA$, then $$E_1\circ\varphi(B)^s=E_1J_g\circ\varphi(A)^s=-J_gE_1\circ\varphi(A)^s=J_g\circ\varphi(A)^s\circ E_1=\varphi(B)^s\circ E_1.$$

    We have that $E_1\circ\varphi(B)^s=-\varphi(B)^s\circ E_1$ for all $B\in\frh$. Therefore $\varphi(B)^s=0$ for all $B\in\frh$.\medskip

    (2) Since $\frg$ is abelian, we have that $\nabla^g_XY=0$ for all $X,Y\in\frg$, and this implies that $E_1\circ\varphi(B)^s=\varphi(E_4B)^s$. Note that $E_1^*=-E_1$ and $J_gE_1=-E_1J_g$. Hence we conclude that $\varphi(A)^s=0$ for all $A\in\frh$ in the same way as in part (1).
\end{proof}

It is worth noting that if we try to mimic Theorem~\ref{thm:construction} starting with two hypersymplectic Lie algebras, then we always obtain a trivial extension. Indeed, let $(\frg,g,J_g,E_g)$ and $(\frh,h,J_h,E_h)$ be hypersymplectic Lie algebras and take $\tilde E$ with $E_1=E_g$, $E_4=E_h$ and $E_2=0=E_3$. Then Theorem~\ref{thm:general_HS} says that $(\tilde\frg,\tilde g,\tilde J,\tilde E)$ is hypersymplectic if and only if the representation $\varphi:\frh\to\Der(\frg)$ defining the semidirect product satisfies \begin{itemize}
    \itemsep 0em
    \item $J_g\circ\varphi(A)^s=\varphi(J_hA)^s$, $J_g\circ\varphi(A)^a=\varphi(A)^a\circ J_g$,
    \item $E_g\circ\varphi(A)^s=\varphi(E_hA)^s$, $E_g\circ\varphi(A)^a=\varphi(A)^a\circ E_g$
\end{itemize} for all $A\in\frh$. We have seen in the proof of Corollary~\ref{cor:HS}~(1) that in this situation $\varphi(A)^s=0$ for all $A\in\frh$, i.e.\ the extension is trivial.\medskip

If we let the $g$-symmetric part of the representation $\varphi$ being zero, that is $\varphi(A)^s=0$ for all $A\in\frh$, then we can find several examples of hypersymplectic Lie algebras.

\begin{corollary}
    If $\frg$ and $\frh$ are abelian hypersymplectic Lie algebras and $\varphi:\frh\to\End(\frg)$ takes values in the Lie algebra $\lie{sp}(2n,\R)$ of $g$-antisymmetric endomorphisms that commute with $J_g$ and $E_g$, then the semidirect product $\frg\rtimes_\varphi\frh$ is hypersymplectic and flat.
\end{corollary}

\begin{proof}
    We set $E_1=E_g$, $E_4=E_h$ and apply Theorem~\ref{thm:general_HS}. The Azencott-Wilson theorem shows that the resulting structure is isometric to a direct product (see Proposition~\ref{prop:pseudoAzencottWilson}), so it is flat.
\end{proof}

\begin{example}\label{ex:nonkodaira}
    A $2$-step nilpotent hypersymplectic Lie algebra is given by $\R^4\rtimes_\varphi\R^4$, where $$\varphi(a_1)=\begin{pmatrix}
        0 & 1 & 0 & 0\\
        0 & 0 & 0 & 0\\
        0 & 0 &0 & 1\\
        0 & 0 & 0 & 0
    \end{pmatrix},\quad\varphi(a_2)=\varphi(a_3)=\varphi(a_4)=0.$$

    In $\R^4$ we have put the flat hypersymplectic structure of Example~\ref{ex:flat_HS_R4}. Its center is given by $\frz(\tilde\frg)=\escal{e_1,e_3,a_2,a_3,a_4}$, so $\tilde\frg$ is not of Kodaira type. Moreover the complex structure $\tilde J$ is not abelian. This example is generalized in Example~\ref{ex:not_Kodaira_type}.
\end{example}

\begin{example}
    More generally, if $\varphi:\R^{4n}\to\lie{sp}(2m,\R)\subset\lie{gl}(4m,\R)$ is any linear map taking values in an abelian subalgebra of nilpotent matrices, we obtain a hypersymplectic structure on the nilpotent Lie algebra $\R^{4m}\rtimes_\varphi\R^{4n}$.
\end{example}

\begin{example}
    Some $2$-step solvable hypersymplectic Lie algebras can be obtained as $\R^4\rtimes_\varphi \R^4$, with $$\varphi(a_1)=\begin{pmatrix}
        1 & 0 & 0 & 0\\
        0 & -1 & 0 & 0\\
        0 & 0 & 1 & 0\\
        0 & 0 & 0 & -1
    \end{pmatrix},\quad\varphi(a_2)=\varphi(a_3)=\varphi(a_4)=0,$$ where on $\R^4$ we have put the hypersymplectic structure of Example~\ref{ex:flat_HS_R4}, or $\R^4\rtimes_\varphi \R^4$, with $$\varphi(a_1)=\begin{pmatrix}
        0 & 1 & 0 & 0\\
        -1 & 0 & 0 & 0\\
        0 & 0 & 0 & 1\\
        0 & 0 & -1 & 0
    \end{pmatrix},\quad\varphi(a_2)=\varphi(a_3)=\varphi(a_4)=0,$$ which is not completely solvable. Another variation is imposing $\varphi(a_4)=\varphi(a_1)$, whilst keeping the others, so that $\ker\varphi$ is non-degenerate.
\end{example}

We can also consider the case where one or two of the starting hypersymplectic Lie algebras are not abelian. The following can be seen as a converse of the first item of Corollary~\ref{cor:HS}.

\begin{corollary}\label{cor:HS_phi^s=0}
    Let $(\frg,g,J_g,E_g)$ and $(\frh,h,J_h,E_h)$ be hypersymplectic Lie algebras. Let $\varphi:\frh\to\Der(\frg)$ be a representation such that $\varphi(A)^s=0$ and $\varphi(A)^a$ commutes with $J_g$ and $E_g$ for all $A\in\frh$. Then $(\tilde\frg,\tilde g,\tilde J,\tilde E)$ is a hypersymplectic Lie algebra.
\end{corollary}

\begin{proof}
    This follows from Theorem~\ref{thm:general_HS} by choosing $\tilde E$ with $E_2=E_3=0$, $E_1=E_g$ and $E_4=E_h$.
\end{proof}

\begin{example}
    Let $\frh=\mathfrak{rh}_3$ with the hypersymplectic structure given in \cite[Theorem~23]{And06} and $\frg=\R^4$ with the hypersymplectic structure as in Example~\ref{ex:flat_HS_R4}. Consider the representation $\varphi:\frh\to\Der(\frg)$ given by $$\varphi(a_1)=\varphi(a_2)=\varphi(a_3)=0,\quad\varphi(a_4)=\begin{pmatrix}
        0 & 1 & 0 & 0\\
        0 & 0 & 0 & 0\\
        0 & 0 &0 & 1\\
        0 & 0 & 0 & 0
    \end{pmatrix}.$$

    Then $\tilde\frg$ is hypersymplectic, $2$-step nilpotent and flat. Its center is given by $\frz(\tilde\frg)=\escal{a_3,e_1,e_3}=[\tilde\frg,\tilde\frg]$, so it is again not of Kodaira type. We also see that $\tilde J$ is not abelian.
\end{example}

We can apply Corollary~\ref{cor:HS_phi^s=0} to get a large class of examples of hypersymplectic Lie algebras which are not of abelian type.

\begin{example}\label{ex:non-abelian_HS}
    Let $(\frg=\R^4,g,J_g,E_g)$ with the flat hypersymplectic structure of Example~\ref{ex:flat_HS_R4} and let $\frh$ be any nilpotent hypersymplectic Lie algebra. Fix a basis $\{a_j\}$ of $\frh$ such that the derived algebra is contained in $\Span{a_j\mid i>1}$, and define the representation $\varphi:\frh\to\Der(\R^4)$ by $$\varphi(a_1)=\begin{pmatrix}
        0 & 1 & 0 & 0\\
        0 & 0 & 0 & 0\\
        0 & 0 &0 & 1\\
        0 & 0 & 0 & 0
    \end{pmatrix}$$ and $\varphi(a_j)=0$ for all $j=2,\ldots,\dim_\R\frh$. The matrix $\varphi(a_1)$ is $g$-antisymmetric and commutes with $J_g$ and $E_g$. Hence, by Theorem~\ref{thm:general_HS}, $\tilde\frg$ is a nilpotent hypersymplectic Lie algebra. It is of non-abelian type in virtue of Lemma~\ref{lemma:abelian_cx}. Moreover, if $\frh$ is non-flat, so is $\tilde\frg$ by Lemma~\ref{lemma:tilde_nabla_XY} (as explained at the beginning of Section~\ref{subsection:non-abelian_h}).
\end{example}

To get some non-flat examples, we can take $\frh$ to be a non-flat $3$-step nilpotent hypersymplectic Lie algebra of \cite{AD06,BS18} or the non-flat $4$-step nilpotent hypersymplectic Lie algebra from \cite[Section~5]{BGL21}.

\begin{example}\label{ex:not_Kodaira_type}
    If in Example~\ref{ex:non-abelian_HS} we take $\frh$ to be abelian, then we obtain that $\tilde\frg$ is a $2$-step nilpotent hypersymplectic Lie algebra which is neither of Kodaira nor abelian type. This can be seen by noticing that the dimension of the center of $\tilde\frg$ is $\dim_\R\frh+1$, which is neither half the dimension of $\tilde\frg$ nor an even number (see Remark~\ref{remark:structureofabeliancomplex}). However, the metric $\tilde g$ is flat. Note that the dimensional argument above implies that no hypersymplectic structure on these Lie algebras can be of Kodaira type.
\end{example}

As a consequence of Example~\ref{ex:not_Kodaira_type}, we get the following result.

\begin{proposition}
    There exist (irreducible) $2$-step nilpotent Lie algebras in any dimension $4n$ for $n\geq2$ which admit a hypersymplectic structure, but admit no abelian hypersymplectic structure and no hypersymplectic structure of Kodaira type.
\end{proposition}

\begin{remark}
Whilst examples of $2$-step nilpotent hypersymplectic Lie algebras which are not of Kodaira type (or direct products of Kodaira type and abelian hypersymplectic Lie algebras), such as those we constructed in Example~\ref{ex:not_Kodaira_type}, are scarce in the literature, we note that one was constructed in \cite[Section~7.2]{CPO11}.
\end{remark}

We can also look for some examples where the representation $\varphi$ has non-zero $g$-symmetric part. Let us consider the following.

\begin{example}
    Let $\frg=\frh=\R^4$ with metric and complex structure given by \begin{align*}
        g&=e^1\otimes e^1+e^2\otimes e^2-e^3\otimes e^3-e^4\otimes e^4,\\
        J_g&=e^1\otimes e_2-e^2\otimes e_1+e^3\otimes e_4-e^4\otimes e_3,
    \end{align*} and $h=-g,J_h=-J_g$. We define the representation $\varphi:\frh\to\Der(\frg)$ by \begin{align*}
        \varphi(a_1)=\varphi(a_3)&=\left(\begin{array}{rrrr}
            1 & 1 & 1 & 1 \\
            1 & -1 & 1 & -1 \\
            -1 & -1 & -1 & -1 \\
            -1 & 1 & -1 & 1
        \end{array}\right),\\
        \varphi(a_2)=\varphi(a_4)&=\left(\begin{array}{rrrr}
            1 & -1 & 1 & -1 \\
            -1 & -1 & -1 & -1 \\
            -1 & 1 & -1 & 1 \\
            1 & 1 & 1 & 1
        \end{array}\right).
    \end{align*}

    The map $\varphi$ satisfies the conditions of Theorem~\ref{thm:construction}, so we have a pseudo-K\"ahler structure on $\tilde\frg$. Moreover, since $\varphi(A)\varphi(B)=0$ for all $A,B\in\frh$, then $\tilde\frg$ is $2$-step nilpotent by Remark~\ref{remark:2-step}. One can check that the endomorphism $$\tilde E=\begin{pmatrix}
        0&\mathbbm{1}\\
        \mathbbm{1}&0
    \end{pmatrix}$$ and this $\varphi$ satisfy all the conditions in Theorem~\ref{thm:general_HS}. Then $(\tilde\frg,\tilde g,\tilde J,\tilde E)$ is a hypersymplectic Lie algebra. Moreover, its center is $\frz(\tilde\frg)=\escal{a_1-a_3,a_2-a_4,e_1-e_3,e_2-e_4}$, which is $\tilde J$-invariant, thus it is of Kodaira type and hence flat.
\end{example}

\begin{remark}
    In \cite[Section~7.3]{CPO11} a similar approach to our construction is considered. Here the authors also obtain examples of hypersymplectic Lie algebras of Kodaira type.
\end{remark}

So far, all the examples of $2$-step nilpotent hypersymplectic Lie algebras in the literature are flat. We have seen this for the ones of Kodaira type (see Proposition~\ref{prop:Kodaira_type_flat}), the ones of abelian type (see Proposition~\ref{prop:2-step_abelian_flat}) and the new examples obtained in this paper. This leads us to make the following conjecture.

\begin{conjecture}\label{con}
    If $(\frg,g,J,E)$ is a $2$-step nilpotent hypersymplectic Lie algebra, then the metric $g$ is flat.
\end{conjecture}

Notice that in \cite{FPPS04}, examples of non-flat hypersymplectic structures are given on some $2$-step nilpotent Lie groups. However these hypersymplectic structures are not left-invariant, so they are not defined on the corresponding Lie algebra.

\subsection{Hypersymplectic structures in every nilpotency step}

As we have mentioned, there are examples of hypersymplectic $k$-step nilpotent Lie algebras for $k=2,3,4$, and as far as we know, all examples in the literature have $k\leq 4$. In this subsection we construct hypersymplectic structures on $k$-step nilpotent Lie algebras for every $k$.\medskip

In relation to Conjecture~\ref{con}, we note that some of the known hypersymplectic Lie algebras of step $k=3,4$ are not flat, and those that we obtain in this section are also non-flat.\medskip

Let $\frg=\R^{4n}$ equipped with the flat hypersymplectic structure $(g,J_g,E_g)$ analogous to the one of Example~\ref{ex:flat_HS_R4}. To be more precise, we consider the following complex and para-complex structure \begin{equation}\label{eq:J_E_R4n}
    J_g=\begin{pmatrix}
    0&-\mathbbm{1}_{2n}\\
    \mathbbm{1}_{2n}&0
\end{pmatrix},\quad E_g=\begin{pmatrix}
    \mathbbm{1}_{2n}&0\\
    0&-\mathbbm{1}_{2n}
\end{pmatrix},\end{equation} together with the metric \begin{equation}\label{eq:g_R4n}
g=\begin{pmatrix} 0 & 0 & 0 & I \\ 0 & 0 & -I & 0 \\
0 & -I & 0 & 0 \\
I & 0 & 0 & 0 \end{pmatrix},\end{equation}
where $I\in\End(\R^n)$ is the anti-diagonal matrix where all the entries are 1, i.e.\ all the entries of $I$ are zero except the entries on the off-diagonal, which are ones. Let $A\in\End(\R^n)$ and consider the endomorphism $$N_{4n}=\begin{pmatrix}
    A&0&0&0\\
    0&-IA^\top I&0&0\\
    0&0&A&0\\
    0&0&0&-IA^\top I
\end{pmatrix}\in\End(\R^{4n}).$$ One can check that such $N_{4n}$ commutes with $J_g$ and $E_g$ and is $g$-antisymmetric.\medskip

Let $A\in\End(\R^n)$ be a $k$-step nilpotent matrix, i.e.\ $A^k=0$ for some $2\leq k\leq n$. Then we have that $N_{4n}^k=0$. By a similar argument as in Example~\ref{ex:non-abelian_HS} we get the following.

\begin{proposition}
\label{prop:arbitrarystep}
    Let $(\frg=\R^{4n},g,J_g,E_g)$ with the flat hypersymplectic structure given by \eqref{eq:J_E_R4n} and \eqref{eq:g_R4n} and let $\frh$ be a $s$-step nilpotent hypersymplectic Lie algebra. Take $s\leq k\leq n$. If we define the representation $\varphi:\frh\to\Der(\frg)$ by $\varphi(a_1)=N_{4n}$ and zero elsewhere, then $\tilde\frg$ is a $k$-step nilpotent hypersymplectic Lie algebra of non-abelian type. Moreover, if $\frh$ is non-flat, so is $\tilde\frg$.
\end{proposition}

If we take $\frh$ to be $\R^4$ or $\mathfrak{rh}_3$ and $k=n$, then we get a flat hypersymplectic structure on an $n$-step nilpotent Lie algebra of dimension $4n+4$. One can ask what is the minimal dimension $4n$ of a nilpotent Lie algebra of step $k$ that admits a hypersymplectic structure. Suggestively, the inequality $k\leq 2n$ is known to hold for $n=1,2$ (see \cite[Corollary~5.15]{BFLM21}). We do not expect our construction to be optimal in regard to this question. For instance, for step $k=4$, an $8$-dimensional hypersymplectic nilpotent Lie algebra was constructed in \cite[Theorem~5.5]{BGL21}, whereas Proposition~\ref{prop:arbitrarystep} yields examples in dimension $20$.


\bibliographystyle{amsplain}
\providecommand{\bysame}{\leavevmode\hbox to3em{\hrulefill}\thinspace}
\providecommand{\MR}{\relax\ifhmode\unskip\space\fi MR }
\providecommand{\MRhref}[2]{%
  \href{http://www.ams.org/mathscinet-getitem?mr=#1}{#2}
}
\providecommand{\href}[2]{#2}

\end{document}